\documentclass[final]{amsart}

\usepackage{Dubois_Graille_Rao}

\begin{document}

\title[Stability of the D1Q3 with relative velocity]{A stability property for a mono-dimensional three velocities scheme with relative velocity}

\author[F.~Dubois]{Fran\c cois Dubois$^{1,2}$}
\author[B.~Graille]{Benjamin Graille$^2$}
\author[S.V.R.~Rao]{S. V. Raghurama Rao$^3$}

\email{benjamin.graille@u-psud.fr}
\address{$^1$ Conservatoire National des Arts et M\'etiers, LMSSC laboratory, Paris, France}
\address{$^2$ Laboratoire de Math\'ematiques d'Orsay, Univ. Paris-Sud, CNRS, Universit\'e Paris-Saclay, 91405 Orsay, France}
\address{$^3$ Department of Aerospace Engineering, Indian Institute of Science, Bangalore - 560 012, India}

\thanks{This work was supported by the Indo-French Centre for Applied Mathematics (IFCAM)}

\date{\today}

\begin{abstract}
In this contribution, we study a stability notion for a fundamental linear one-dimensional lattice Boltzmann scheme, this notion being related to the maximum principle. We seek to characterize the parameters of the scheme that guarantee the preservation of the non-negativity of the particle distribution functions.
In the context of the relative velocity schemes, we derive necessary and sufficient conditions for the non-negativity preserving property.
These conditions are then expressed in a simple way when the relative velocity is reduced to zero. For the general case, we propose some simple necessary conditions on the relaxation parameters and we put in evidence numerically the non-negativity preserving regions.
Numerical experiments show finally that no oscillations occur for the propagation of a non-smooth profile if the non-negativity preserving property is satisfied. 
\end{abstract}


\maketitle


\section{Introduction} 
\label{sec:intro}

Studying stability of lattice Boltzmann schemes is a  non-trivial problem.
Classically for this purpose, the scheme is linearized around a constant state and a Fourier analysis is performed.  We refer to the work of Lallemand and Luo  \cite {lallemand_theory_2000} for the D2Q9 scheme applied to hydrodynamics.  Note also the work of Ginzburg
{ \it et al.} \cite {ginzburg_truncation_2012} extending the Fourier analysis to a wide variety of different two and three dimensional lattice Boltzmann schemes. Instabilities and their interpretattion in terms of bulk viscosity has been proposed by Dellar \cite{dellar_nonhydrodynamic_2002}. But no mathematical analysis has been performed. 

A new way of improving stability is proposed by Geier \cite{geier_cascaded_2006}, who proposed a new generalized lattice Boltzmann scheme with the approach of relative velocities and utilized it for hydrodynamics applications~\cite{dubois_lattice_2015}.
The tentative of analysis of this method for   a two-dimensional scalar linear equation has been also proposed~\cite{dubois_stability_2015}. 

Even if it is a difficult task, it is well known that Fourier analysis is not the best method for analyzing nonlinear hyperbolic equations. Total variation diminishing schemes, developed for suppressing oscillations in higher order CFD algorithms, provide an alternative nonlinear stability analysis tool for analysing the schemes for nonlinear wave propagation. The convergence of such schemes is well established \cite{osher_muscl_1985}.
The underlying stability notion concerns the maximum principle.
This notion can be extended to nonlinear cases and a first attempt has been proposed in  \cite {caetano_result_2019} for lattice Boltzmann schemes for the D1Q2 scheme for scalar nonlinear hyperbolic equations. 

In this contribution, we propose to investigate the stability in the maximum sense, of a linear mono-dimensional lattice Boltzmann scheme with three velocities.  More precisely, we look to a positivity constraint for a particle distribution function 
in the context of a relative velocities. 

In Section~2, we describe the scheme and the underlying advection model. More precisely, the local relaxation step is written as a linear operator on the particle distribution functions. If all the coefficients of the underlying matrix are nonnegative, the non-negativity of the distribution is maintained during this step. Because the transport step is just a change of locus, the non-negativity is maintained for the  whole time step of the scheme.  The question is then to find appropriate conditions to handle this property.
In Section~3, a necessary and sufficient condition is derived on the parameters in order to ensure that the scheme has the stability property.
In Section~4, we completely describe the classical case  where the relative velocity is reduced to zero. 
In Section~5, the general case is presented. With an analytical study for necessary conditions and a numerical one for a complete description of the stability zones.
In Section~6, numerical experiments show the correlation of the positivity constraint for a particle distribution and the presence of oscillations for discontinuous profiles.



\section{Description of the framework}
\label{sec:problem}

\subsection{Description of the scheme}

In this contribution, we investigate a mono-dimensional three velocities linear lattice Boltzmann scheme with relative velocity \cite{dubois_lattice_2015}. 
Denoting \(\Delta x\) the spatial step, \(\Delta t\) the time step, and \(\lambda=\Delta x/\Delta t\) the lattive velocity, 
this scheme can be described in a generalized D.~d'Humi\`ere's framework \cite{dhumiere_generalized_1992}: 
\begin{enumerate}[(1)]
\item the 3 velocities \(c_1=-1\), \(c_2=0\), and \(c_3=1\);
\item the 3 associated distributions \(f_1\), \(f_2\), and \(f_3\);
\item the 3 moments \(\rho\), \(q(u)\), and \(\varepsilon(u)\) given by
\begin{equation*}
 \rho = \smashoperator{\sum_{1\leq j\leq 3}} f_j,
 \quad
 q(u) = \lambda \smashoperator{\sum_{1\leq j\leq 3}} (c_j-u) f_j,
 \quad
 \varepsilon(u) = 3 \lambda^2 \smashoperator{\sum_{1\leq j\leq 3}} (c_j-u)^2 f_j 
 - 2\lambda^2 \smashoperator{\sum_{1\leq j\leq 3}} f_j,
\end{equation*}
where \(u\) is a given scalar representing the relative velocity;
\item the equilibrium value of the 3 moments
\begin{equation*}
 \rho^\eq = \rho,
 \quad
 q^\eq(u) = \lambda (V-u)\rho,
 \quad
 \varepsilon^\eq(u) = \lambda^2 (3u^2-6uV+\alpha) \rho,
\end{equation*}
where \(V\) and \(\alpha\) are given scalars (without loose of generality, we assume that \(V>0\));
\item the 2 relaxation parameters \(s\) and \(s'\) such that the relaxation phase reads
\begin{equation*}
 q^\star(u) = (1-s) q(u) + s q^\eq(u),
 \quad
 \varepsilon^\star(u) = (1-s') \varepsilon(u) + s' \varepsilon^\eq(u).
\end{equation*}
\end{enumerate}

In this formalism, the moments are defined as polynomial functions of the discrete velocities and the discrete distribution functions.
Indeed, introducting \(P_1=1\), \(P_2=\lambda X\), and \(P_3=\lambda^2 (3X^2-2)\), the three moments read
\begin{equation*}
\rho = \sum_{1\leq j\leq 3} P_1(c_j-u) f_j,
\quad
q(u) = \sum_{1\leq j\leq 3} P_2(c_j-u) f_j,
\quad
\varepsilon(u) = \sum_{1\leq j\leq 3} P_3(c_j-u) f_j.
\end{equation*}

The equilibrium values are chosen such that the equilibrium distributions do not depend on the relative velocity \(u\). Indeed, we have:
\begin{equation*}
 f_j^\eq = \tfrac{1}{6} \rho \bigl(
 2 + 3c_jV + (3c_j^2-2)\alpha
 \bigr), 
 \quad 1\leq j\leq 3.
\end{equation*}
Note that this scheme can be used (see {\it e.g.}  \cite {dubois-lallemand_2009}) to simulate a scalar transport equation with constant velocity \(\lambda V\) given by
\begin{equation*}
 \partial_t\rho + \lambda V \partial_x\rho = 0.
\end{equation*}

Denoting \(x_k=k\Delta x\), \(k\in{\mathbb Z}\), the discret spatial points and \(t^n=n\Delta t\), \(n\in{\mathbb N}\), the discret time points, one time step of the scheme reads
\begin{equation*}
 f_j(t^n + \Delta t, x_k+c_j \Delta t) 
 = f_j(t^{n+1}, x_{k+j})
 = f_j^\star(t^n,x_k),
 \quad
 1\leq j\leq 3.
\end{equation*}

\subsection{Matrix notation for the relaxation step}

In this contribution, we are concerned with the {\em non-negativity of the particle distribution functions}. This property can be viewed as refering to a weak  maximum principle for linear schemes. Indeed, it is always possible, by adding a constant, to assume that all the particle distribution functions are initialy nonnegative. If the scheme ensures that this property of non-negativity remains  as time marches, each particle distribution function is then bounded, as their total sum  is conserved. Moreover, as the transport step consists simply in exchanging the position of the particle distribution functions, we focus on the relaxation step.

We use a matrix notation for the relaxation step as it can be read as a multiplication by a matrix. As this step is local in space, we omit the dependancy on time and space. We define the vector of the distribution functions \(\vectF\)
\begin{equation*}
    \vectF = ( f_1 \; f_2 \; f_3 )^T.
\end{equation*}
One relaxation step then reads
\begin{equation*}
 \vectF^\star = \matR \vectF,
 \quad
\end{equation*}
where the matrix \(\matR\) is defined by
\begin{equation*}
    \matR = \matiM \matiT \Bigl(
        \matid + \matS \bigl(
            \matT\matE\matiT-\matid
        \bigl)
    \Bigr) \matT \matM,
\end{equation*}
with
\begin{gather*}
    \matM = \begin{pmatrix}
        1 & 1 & 1\\
        -\lambda & 0 & \lambda\\
        \lambda^2 & -2\lambda^2 & \lambda^2
    \end{pmatrix},
    \quad
    \matT = \begin{pmatrix}
        1 & 0 & 0\\
        -\lambda u & 1 & 0\\
        3\lambda^2 u^2 & -6\lambda u & 1
    \end{pmatrix},
    \\
    \matS = \begin{pmatrix}
        0&0&0\\ 0&s&0 \\ 0&0&s'
    \end{pmatrix},
    \quad
    \matE = \begin{pmatrix}
        1 & 0 & 0 \\
        V \lambda & 0 & 0 \\
        \alpha \lambda^2 & 0 & 0
    \end{pmatrix},
    \quad
    \matid = \begin{pmatrix}
        1&0&0 \\ 0&1&0 \\ 0&0&1
    \end{pmatrix}.
\end{gather*}

The coefficients of the matrix \(\matM\) are obtained by the relations \(\matM_{k,j} = P_k(c_j)\), \(1\leq k,j\leq 3\) and those of the matrix \(\matT\) by the change of basis formula:
\(\matT_{k,l}\) is the coefficient of the \(l^{\text{th}}\)-element \(P_l(c_j)\) in the definition of \(P_k(c_j-u)\) according to 
\begin{equation*}
P_k(c_j-u) = \sum_{1\leq l\leq 3} \matT_{k,l} P_l(c_j), \qquad 1\leq k, j\leq 3.
\end{equation*}
The matrix \(\matM\) is then the change of basis that transforms the vector \(\vectF\) into the vector \(\vectMz = (\rho, q(0), \varepsilon(0))^T\):
\begin{equation*}
    \vectMz = \matM \vectF, \qquad \vectM = \matT \matM \vectF.
\end{equation*}
The matrix \(\matT\) can then be viewed as the change of basis matrix from the classical moments without relative velocity toward the moments with relative velocity.

\subsection{Remark on the choice of the moments}

Note that the last moment \(\varepsilon(u)\) that is chosen in this contribution is not the energy but a moment that is orthogonal to the two first ones, \(\rho\) and \(q(u)\).
In this section, we show that all the results of the contribution would be identical by choosing the last moment as the energy: the relaxation matrix \(\matR\) would  still  be the same.   

We consider two schemes with two different choices of polynomials: the moments of the first scheme are defined by \((P_1,P_2,P_3)\) while the moments of the second scheme by \((\widehat P_1,\widehat P_2,\widehat P_3)\).
The first moment is the same in both schemes to be able to simulate the same transport equation. We then have \(\widehat P_1=P_1\).
We define \(\matC\) the change of basis matrix associated to the tranformation \(\matM\) into \(\matMh\): 
\begin{equation*}
    \matMh = \matC \matM.
\end{equation*}
The first line of \(\matC\) is then \((1, 0, 0)\).

\begin{proposition}
    We assume that the equilibrium values of the distribution functions are the same, that is \(\matEh=\matC\matE\), and that the relaxation parameters are the same, that is \(\matSh=\matC\matS\).
    Then, we have \(\matRh = \matR\) for all \((s,s')\) iff 
    \begin{equation*}
        \widehat P_2 \in \spn(P_1, P_2),
        \quad
        \widehat P_3 \in \spn(P_1, P_3),
        \quad
        \text{in}
        \quad
        \mathbb{R}[X]/X(X-1)(X+1). 
    \end{equation*}
\end{proposition}

\begin{proof}
    First, we immediately obtain the following relations by identifying the coefficients of \(\matMh\) and of \(\matM\):
    \begin{equation*}
        \widehat{P}_k(c_j) = \sum_{1\leq l\leq 3} \matC_{k,l} P_l(c_j), 
        \qquad 1\leq j\leq 3.
    \end{equation*}
    We deduce that 
    \begin{equation*}
        \matTh = \matC\matT\matiC.
    \end{equation*}
    Moreover, as the second and the third columns of \(\matE\) are zero (the equilibrium values depend only on the first moment \(\rho\)), and as the first line of \(\matC\) is then \((1, 0, 0)\), we have
    \begin{equation*}
        \matE\matC = \matE.
    \end{equation*}
    We have
    \begin{align*}
        \matRh &= \matiMh \matiTh \Bigl(
            \matid + \matSh \bigl(
                \matTh\matEh\matiTh-\matid
            \bigl)
        \Bigr) \matTh \matMh \\
        &= \matiM\matiT\matiC \Bigl(
            \matid + \matS \bigl(
                \matC\matT\matE\matC\matiT\matiC - \matid
            \bigr)
        \Bigr) \matC\matT\matM \\
        &= \matiM\matiT \Bigl(
            \matid + \matiC\matS\matC \bigl(
                \matT\matE\matiT - \matid
            \bigr)
        \Bigr) \matT\matM.
    \end{align*}
    Then
    \begin{align*}
        \matRh - \matR &=
        \matiM\matT \matiC \bigl(
            \matS\matC - \matC\matS
        \bigr) \matT \bigl(
            \matE - \matid
        \bigr) \matM.
    \end{align*}
    As the matrices \(\matM\), \(\matT\), and \(\matC\) are invertible, the condition \(\matRh=\matR\) is equivalent to 
    \((\matS\matC - \matC\matS) \matT (\matE - \matid)=0\). Denoting
    \begin{equation*}
        \matC = \begin{pmatrix}
            1&0&0\\ c_{21}&c_{22}&c_{23} \\ c_{31}&c_{32}&c_{33}
        \end{pmatrix},
    \end{equation*}
    a straightforward calculation yields
    \begin{equation*}
        (\matS\matC - \matC\matS) \matT (\matE - \matid) = (s-s')\begin{pmatrix}
            0&0&0\\
            c_{23}\lambda^2(\alpha-6V)&
            c_{23}6\lambda u&
            -c_{23}\\
            c_{32}\lambda V&-c_{32}&0
        \end{pmatrix}.
    \end{equation*}
    Then the property \(\matRh=\matR\) for all values of \(s\) and \(s'\) is equivalent to \(c_{23}=c_{32}=0\), that ends the proof.
\end{proof}

\subsection{Positivity of the \(\matR\) matrix}

The velocity \(V\) being fixed, we propose to give a full description of the sets
\begin{equation*}
 \Omega^{V, u} = \Bigl\lbrace
 (s, s', \alpha) \in \mathbb{R}^4 \textrm{ such that } \matR \textrm{ is a non-negative matrix}
 \Bigr\rbrace,
 \quad u\in\mathbb{R}.
\end{equation*}
Indeed, the non-negativity of the matrix \(\matR\) imposes that all the distributions \(f_j\), \(1\leq j\leq3\), remain non-negative if they are so at the initial time. These sets are first described by a set of nine inequalities that can be joined into just one. Numerical illustrations are then given to visualize it in the characteristic cases including SRT, MRT, and relative velocity scheme.



\section{Positivity of the iterative matrix}
\label{sec:inequalities}

The nine inequalities obtained from the matrix \(\matR\) can be combined neatly into one formula.

The inequalities are
\begin{equation} \label{Rij-positive} \left \{ 
    \begin{aligned}
    R_{0,0} &= V s u - \tfrac12 V s - V s' u + \tfrac16 \alpha s' + s u - \tfrac12 s - s' u - \tfrac16 s' + 1 &\geqslant 0, \\ 
    R_{0,1} &= V s u - \tfrac12 V s - V s' u + \tfrac16 \alpha s' + \tfrac13 s' &\geqslant 0, \\ 
    R_{0,2} &= V s u - \tfrac12 V s - V s' u + \tfrac16 \alpha s' - s u + \tfrac12 s + s' u - \tfrac16 s' &\geqslant 0, \\ 
    R_{1,0} &= - 2 V s u + 2 V s' u - \tfrac13 \alpha s' - 2 s u + 2 s' u + \tfrac13 s' &\geqslant 0, \\ 
    R_{1,1} &= - 2 V s u + 2 V s' u - \tfrac13 \alpha s' - \tfrac13 2 s' + 1 &\geqslant 0, \\ 
    R_{1,2} &= - 2 V s u + 2 V s' u - \tfrac13 \alpha s' + 2 s u - 2 s' u + \tfrac13 s' &\geqslant 0, \\ 
    R_{2,0} &= V s u + \tfrac12 V s - V s' u + \tfrac16 \alpha s' + s u + \tfrac12 s - s' u - \tfrac16 s' &\geqslant 0, \\ 
    R_{2,1} &= V s u + \tfrac12 V s - V s' u + \tfrac16 \alpha s' + \tfrac13 s' &\geqslant 0, \\ 
    R_{2,2} &= V s u + \tfrac12 V s - V s' u + \tfrac16 \alpha s' - s u - \tfrac12 s + s' u - \tfrac16 s' + 1 &\geqslant 0. 
    \end{aligned} \right.   
\end{equation}
We prove now that the previous nine inequalities can be written in a much more  lucid way.

\begin{proposition}
We introduce the reduced parameters  \(\overline{u}\) and \(\gamma\) according to
\begin{equation} \label{ubar-gamma} 
  \overline{u} = 2 u (s - s') ,
  \quad
  \gamma = \frac{s'}{6} (1-\alpha) - u (s - s')V . 
\end{equation}
Then the nine previous inequalities \(R_{i,j} \geq 0\) displayed in \eqref{Rij-positive}  are equivalent to 
\begin{equation} \label{cns-stabilite}
\max ( s'-1 ,| \overline{u} | )  
\leq 
2 \gamma  
\leq 
\min ( 2 - s - | \overline{u} -   s V | ,
s - | \overline{u} + s V | , s' - | s V | ) .
\end{equation}
\end{proposition}

\begin{proof}
Consider first the two inequalities associated with \(R_{0,0}\) and  \(R_{2,2}\):
\begin{equation*} 
  \left\lbrace
  \begin{aligned} 
  V s u  - V s' u + \tfrac16 \alpha s' - \tfrac12 s - \tfrac16 s' + 1 
  &\geqslant \phantom{-}\tfrac12 V s - s u + s' u ,\\ 
  V s u  - V s' u + \tfrac16 \alpha s' - \tfrac12 s - \tfrac16 s' + 1
  &\geqslant -\tfrac12 V s + s u - s' u .  
  \end{aligned}
  \right.
\end{equation*}
They can be synthetized in the following form
\begin{equation*}
  \tfrac12
  \left\vert
    \overline{u} - V s
  \right\vert
  \leq 1 - \tfrac12 s  - \bigl( 
    \tfrac16 s' (1-\alpha) + u (s'-s) V \bigr)  
  = 1 - \tfrac12 s - \gamma 
\end{equation*}
and we can write this relation as
\begin{equation} \label{provis-01} 
2 \gamma \leq 2 - s  - \vert \overline{u} - s V \vert. 
\end{equation}

Write now the inequalities \eqref{Rij-positive}  associated with \(R_{0,1}\) and  \(R_{2,1}\):
\begin{equation*}
  \left\lbrace
  \begin{aligned}
    V s u - V s' u + \tfrac16 \alpha s' + \tfrac13 s'
    & \geqslant \phantom{-}\tfrac12 V s ,\\ 
    V s u - V s' u + \tfrac16 \alpha s' + \tfrac13 s'
    & \geqslant - \tfrac12 V s. 
  \end{aligned}
  \right.
\end{equation*}
In other words,  \(\vert V s/2 \vert \leq - \gamma +  s'/2 \). Then 
\begin{equation} \label{provis-02} 
2 \gamma \leq s'  - \vert s V \vert. 
\end{equation}

We now focus on the inequalities \eqref{Rij-positive} associated with \(R_{0,2} \) and  \(R_{2,0}\):
\begin{equation*}
  \left\lbrace
  \begin{aligned}
    V s u - V s' u + \tfrac16 \alpha s' + \tfrac12 s - \tfrac16 s'
    & \geqslant \phantom{-}\tfrac12 V s +  s u  - s' u ,\\ 
    V s u - V s' u + \tfrac16 \alpha s' + \tfrac12 s - \tfrac16 s'
    & \geqslant - \tfrac12 V s - s u  + s' u .
  \end{aligned}
  \right.
\end{equation*}
We have \(\vert V s/2 +  s u  - s' u \vert \leq V s u - V s' u + \alpha s'/6 + s/2 - s'/6 = s/2 - \gamma \). In consequence, 
\begin{equation} \label{provis-03} 
2 \gamma \leq s  - \vert \overline{u} + s V \vert . 
\end{equation}

Considering the  inequalities \eqref{Rij-positive} with \(R_{1,0}\) and  \(R_{1,2}\), we have 
\begin{equation*}
  \left\lbrace
  \begin{aligned}
    - V s u +  V s' u - \tfrac16 \alpha s' + \tfrac16 s'
    & \geqslant   \phantom{-} s u -  s' u ,\\ 
    - V s u +  V s' u - \tfrac16 \alpha s' + \tfrac16 s'
    & \geqslant   -  s u +  s' u ,
  \end{aligned}
  \right.
\end{equation*}
and \(\vert\overline{u}/2\vert \leq - V s u +  V s' u - \alpha s'/6 + s'/6 = \gamma\).
In consequence,
\begin{equation} \label{provis-04} 
  \vert\overline{u}\vert \leq 2 \gamma .
\end{equation}

The last inequality \(R_{1,1} \geq 0\) can be written as
\(- V s u + V s' u - \alpha s'/6 - s'/3 + 1/2 \geqslant 0 \)
and this inequality is equivalent to \(\gamma - s'/2 + 1/2 \geqslant 0\). 
In other terms,
\begin{equation} \label{provis-05} 
s' - 1 \leq 2 \gamma .
\end{equation}

The inequalities \eqref{provis-01}, \eqref{provis-02}, and  \eqref{provis-03} establish a triple majoration of \(2\gamma\) whereas the inequalities  \eqref{provis-04} and  \eqref{provis-05} show a double minoration of the same quantity. The proof is completed. 
\end{proof}



\section{The particular case $u=0$}
\label{sec:norelative}

In this section, we suppose that the relative velocity \(u\) is reduced to zero. Then the necessary and sufficient conditions \eqref{cns-stabilite} for the stability can be written as
\begin{equation} \label{cns-stabilite-uzero} 
\max ( s'-1 ,0)  
\leq \frac{s'}{3} (1-\alpha) \leq 
\min ( 2 - s - \vert s V \vert, s - \vert s V \vert, s' - \vert s V \vert ) . 
\end{equation}

\begin{proposition}
To fix the ideas, we suppose  that the advection velocity \(V\) is non-negative:
\begin{equation} \label{V-positive} 
  V \geq 0 .
\end{equation}
The case \(V \leq 0 \) follows directly. 
When \(u = 0\), the  reduced stability conditions
\begin{equation} \label{reduced-cns-stabilite-uzero} 
  \max ( s'-1 ,0)  
  \leq
  \min ( 2 - s - \vert s V \vert, s - \vert s V \vert, s' - \vert s V \vert ) . 
  \end{equation}
are equivalent to the following conditions for the relaxation parameters 
\begin{equation} \label{cns-stabilite-uzero-explicites} 
  \left\lbrace
  \begin{aligned}
    &0  \leq s , s' \leq 2, \\ 
    &s' \geq s V, \\ 
    &s  \leq 2/(1+V),\\ 
    &s' \leq \min ( 3 - (1+V) s, 1 + (1-V) s ) 
  \end{aligned}
  \right.
\end{equation}
joined with a natural Courant type condition for explicit schemes on the advection velocity
\begin{equation} \label{V-leq-1_u-0} 
  V \leq 1 .
\end{equation}
\end{proposition}

Of course,  the conditions \eqref{cns-stabilite-uzero} have still to be imposed for the equilibrium parameter \(\alpha\) when the pair \(s,s'\) is given. In particular,
\begin{equation} \label{alpha-leq-1} 
  \alpha \leq 1
\end{equation}
and 
\begin{equation} \label{sple3salp2} 
 s' \leq \frac{3}{\alpha + 2 } .  
\end{equation}

\begin{figure}
  \includegraphics[width=.32 \textwidth, angle=0]{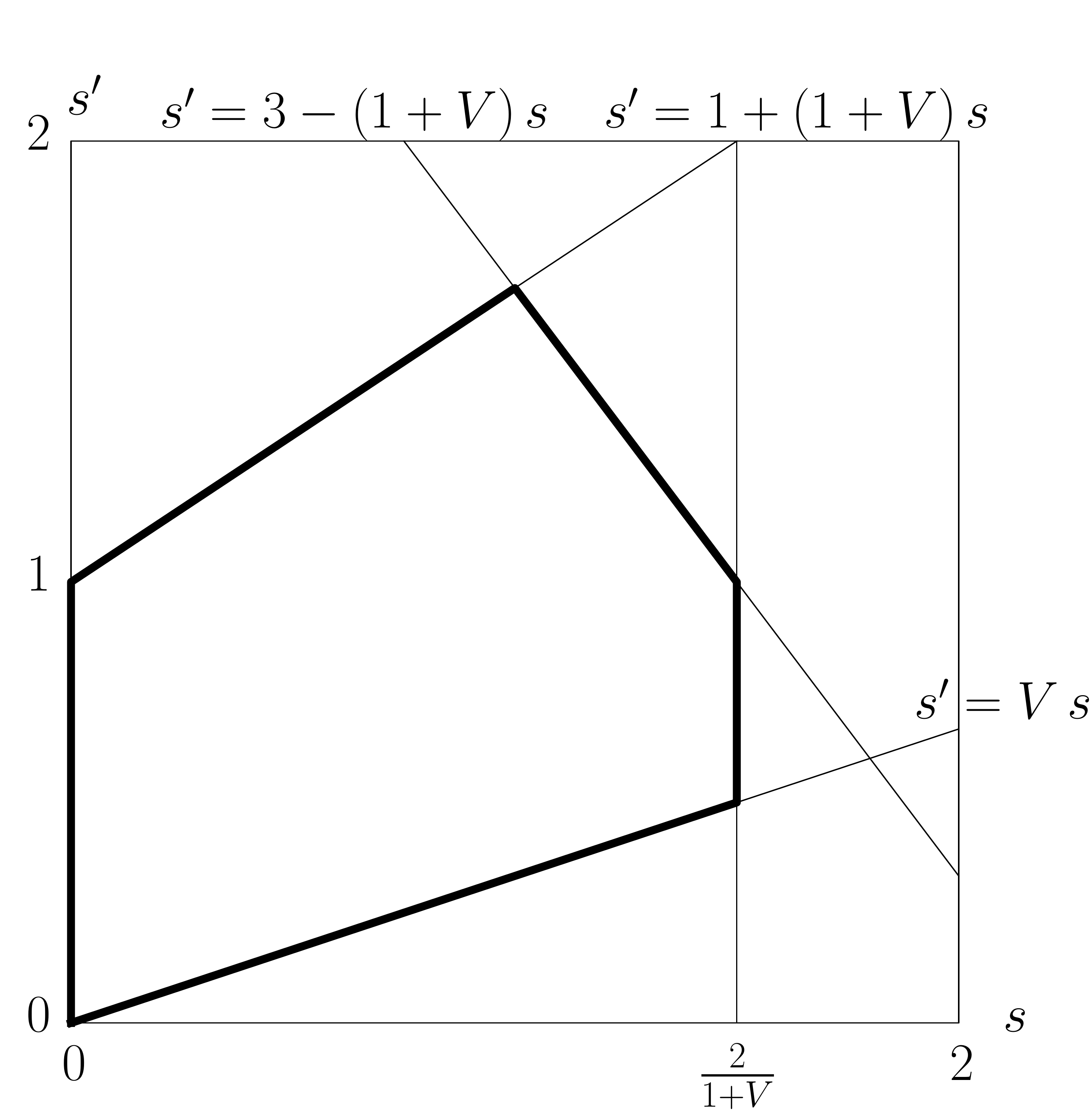}
  \caption{Necessary and sufficient stability regions described by the inequalities \eqref{reduced-cns-stabilite-uzero} for
  a null relative velocity \(u\). Illustration proposed for  \(V = 2/3\). }
  \label{cns-uzero-fig}
\end{figure}

\begin{proof}
We first observe that 
\(0 \leq \max ( s'-1 , 0) \leq \vert s V \vert \leq \min  (2-s , s) \).
Then \(0 \leq s \leq 2\).
Secondly, we have \(\vert s V \vert \leq s' \) and because both \(s\) and \(V\) are positive, we have \(0 \leq  s V \leq s'\). We have also \(s V \leq  s \) and \eqref{V-leq-1_u-0} is established.

Moreover, \(s \geq 0 \) and  \(\vert s V \vert \leq 2-s \) implies \(s \leq 2/(1+V)\). 
Due to the positivity of the parameter \(s\), we deduce from \eqref{cns-stabilite-uzero} a new set of inequalities: 
\(s' - 1 \leq \max ( s'-1 ,0) \leq \min ( 2-s - s V , s - s V )\).
Then \(s'\leq\min(3-(1+V)s, 1 + (1-V) s\).
Moreover, \(s'\leq 3- (1+V) s \leq 2 \) because \(V \geq 0\). 

Conversely, if the relations \eqref{cns-stabilite-uzero-explicites} and \eqref{V-leq-1_u-0} are satisfied, we have \(s'-1 \leq 2 - s - s V\),  \(s'-1 \leq s - s V \) and \(s'-1 \leq s' - s V \). Then \(s'-1 \leq \min (2 - s - s V , s - s V ,s' - s V \). Moreover, \(0 \leq 2 - s - s V\), \(0 \leq s - s V \) because \(V \leq 1 \) and \(0 \leq s' - s V\). Thus \(0 \leq\min(2 - s - s V ,s - sV ,s' - s V\).

Finally the inequalities \eqref{reduced-cns-stabilite-uzero} are established and the proposition is proven.
\end{proof}



\section{The general case}
\label{sec:relative}

We analyse the general case of non-zero $u$ in this section, with suitable illustrations of the stability region for various ranges of the parameters.

\subsection{Necessary conditions for stability}
\label{sec:cn}

In this subsection, we prove the following proposition. 

\begin{proposition}\label{th:necessary}
We suppose that the necessary and sufficient stability conditions \eqref{cns-stabilite} are satisfied:
\begin{equation*} 
  \max ( s'-1 ,| \overline{u} | )  
  \leq 
  2 \gamma  
  \leq 
  \min ( 2 - s - | \overline{u} -   s V | ,
  s - | \overline{u} + s V | , s' - | s V | )
\end{equation*}
with the notations \eqref{ubar-gamma}:
\begin{equation*}
  \overline{u} = 2 u (s - s') ,
  \quad
  \gamma = \frac{s'}{6} (1-\alpha) - u (s - s')V . 
\end{equation*}
We suppose also that the advection velicity \(=V\) is positive. 
Then if the scheme is stable, the point \((s, s')\) satisfies the following inequalities
\begin{equation} \label{necessary} 
  \left \lbrace 
  \begin{aligned}
    &0 \leq s V \leq s' \leq 2 \\
    &0 \leq s V \leq 1 \\
    &0 \leq s \leq 2 \\
    &s' \leq \min ( 2 - sV , s + 1 , 3-s ) \\
    &s  \leq 2/(1+V)  . 
  \end{aligned}
  \right.
\end{equation}
With these necessary stability conditions, the parameter \(\overline{u}\) has been eliminated. We have also, necessarily 
\begin{equation} \label{V-leq-1}   
  V \leq  1
\end{equation} 
and
\begin{equation} \label{abs-ubar} 
  \vert \overline{u} \vert \leq \frac{1}{2} .
\end{equation}
\end{proposition}

\begin{proof}
We start from the inequalities \eqref{cns-stabilite}. Then we have
\(0 \leq \vert\overline{u}\vert \leq \max ( s'-1 ,\vert\overline{u}\vert) \leq \min ( 2 - s - \vert\overline{u} - sV\vert , s - \vert\overline{u} + sV\vert, s' - \vert sV\vert) \leq s' - \vert sV\vert\)
and \(s' \geq\vert sV\vert \geq sV\).
With the same type of inequalities,
\(\vert\overline{u}\vert + \vert\overline{u} - sV\vert \leq 2 - s\) and \(\vert \overline{u}\vert + \vert\overline{u} + sV\vert \leq s\).

If we remark that
\(\vert\overline{u}\vert\leq\vert\overline{u} - sV\vert/2 +  \vert\overline{u} + sV\vert/2\), we deduce that \(\vert\overline{u}\vert\leq(2 - s + s )/2 = 1 \) and the inequality \eqref{abs-ubar} is proven. 

We have the triangular inequality  \(\vert sV\vert \leq \vert\overline{u} - sV\vert + \vert\overline{u}\vert\). Then from the genaral stability conditions \eqref{cns-stabilite}, we deduce \(\vert\overline{u} - sV\vert + \vert\overline{u}\vert\leq 2 - s\) and \(\vert sV\vert \leq 2 - s\).
In a similar way, 
\(\vert sV \vert \leq \vert \overline{u} + sV \vert + \vert \overline{u}\vert \), \(\vert \overline{u} + sV \vert + \vert \overline{u} \vert \leq s \) from   \eqref{cns-stabilite} and finally \(\vert sV\vert \leq 2 - s\).
We put together the two inequalities and we have \(0 \leq \vert sV \vert \leq \min ( s,2-s)\).

In consequence, we have \(0\leq s \leq 2\) and the third  point is proven. Since we made the choice of  \(V \geq 0\) then \(s \geq 0\). The first inequality of the two first points are established. From \(sV \leq s \) we have \(V \leq 1\) and the relation \eqref{V-leq-1} is true. Moreover,  \(sV \leq 2-s \) and the last inequality of \eqref{necessary} is true. 

Consider now the inequalities
\(s'-1  \leq 2\gamma \leq \min ( 2 - s - \vert \overline{u} - sV \vert ,
s - \vert \overline{u} + sV \vert ) \). We deduce \(s'-1 + \vert \overline{u} - sV\vert \leq 2-s \) and \(s'-1 + \vert \overline{u} + sV \vert \leq s \). Due to the positvity of the absolute values, we have also \(s'\leq 3 - s \) and  \(s'\leq s+1 \). A part of the fifth inequality of \eqref{necessary} is proven.

Finally, due to the triangular inequality, \(sV = \vert sV \vert \leq \vert \overline{u} - sV \vert/2 + \vert \overline{u} + sV \vert/2 \). We add this inequality with the two following ones: \(s'-1 + \vert \overline{u} - sV \vert \leq 2-s \) and \(s'-1 + \vert \overline{u} + sV \vert \leq s \). Then \(s'-1 + sV \leq (2-s + s)/2 = 1 \) and \(s' \leq 2 - sV \). The fifth inequality of \eqref{necessary} is completely established. Because \(sV \geq 0 \), we have also \(s' \leq2 \) and the first inequality of \eqref{necessary} is also established. 
\end{proof}

We illustrate the zones of necessary stability in the Figure~\ref{cn-fig} for
five particular velocities: \(V \in\lbrace 0, 1/4, 1/3, 1/2, 2/3, 1\rbrace\). 

\begin{figure}[H]
  \includegraphics[width=.49 \textwidth, angle=0] {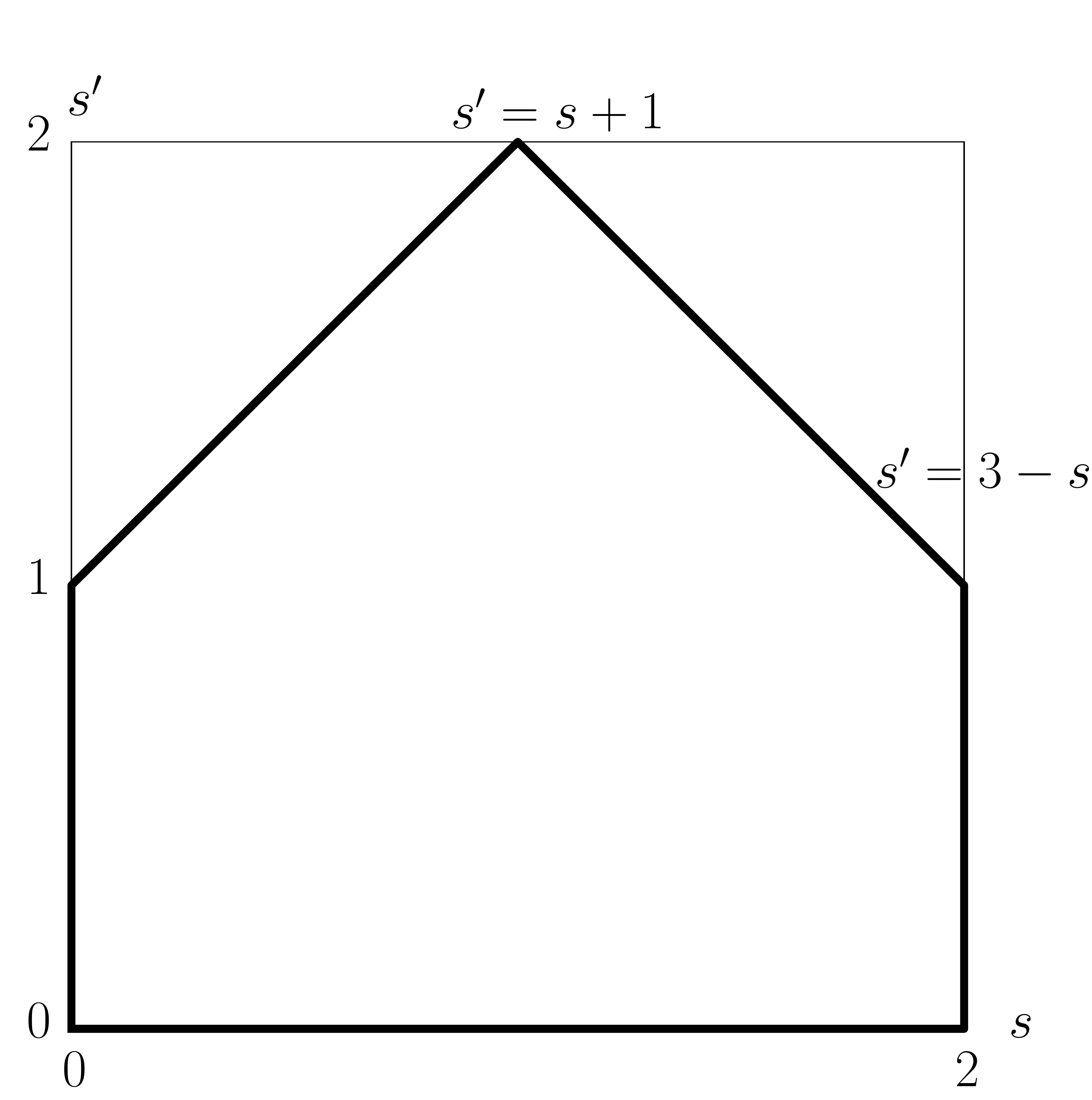}
  \hfill 
  \includegraphics[width=.49 \textwidth, angle=0] {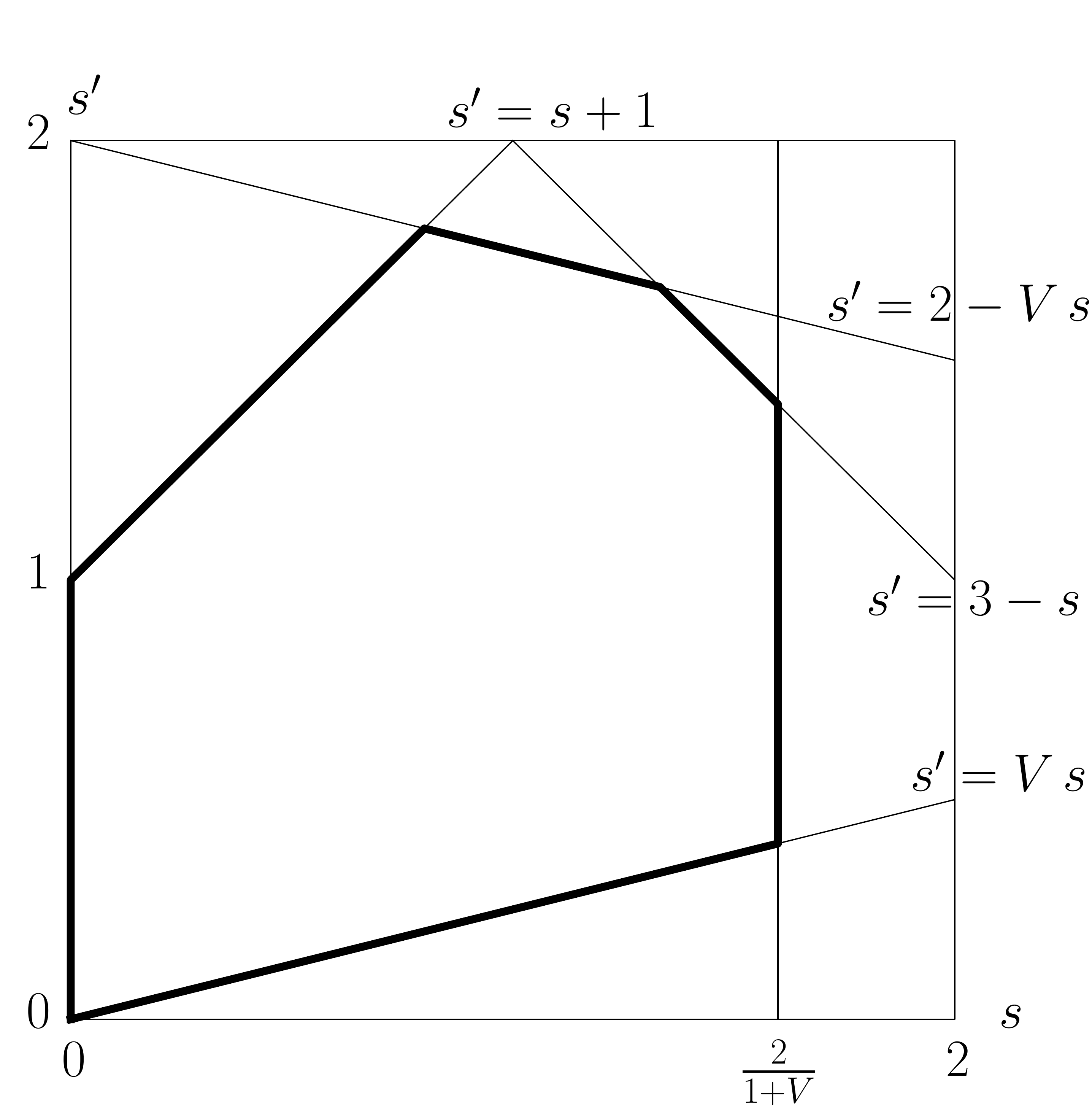}
  \includegraphics[width=.49 \textwidth, angle=0] {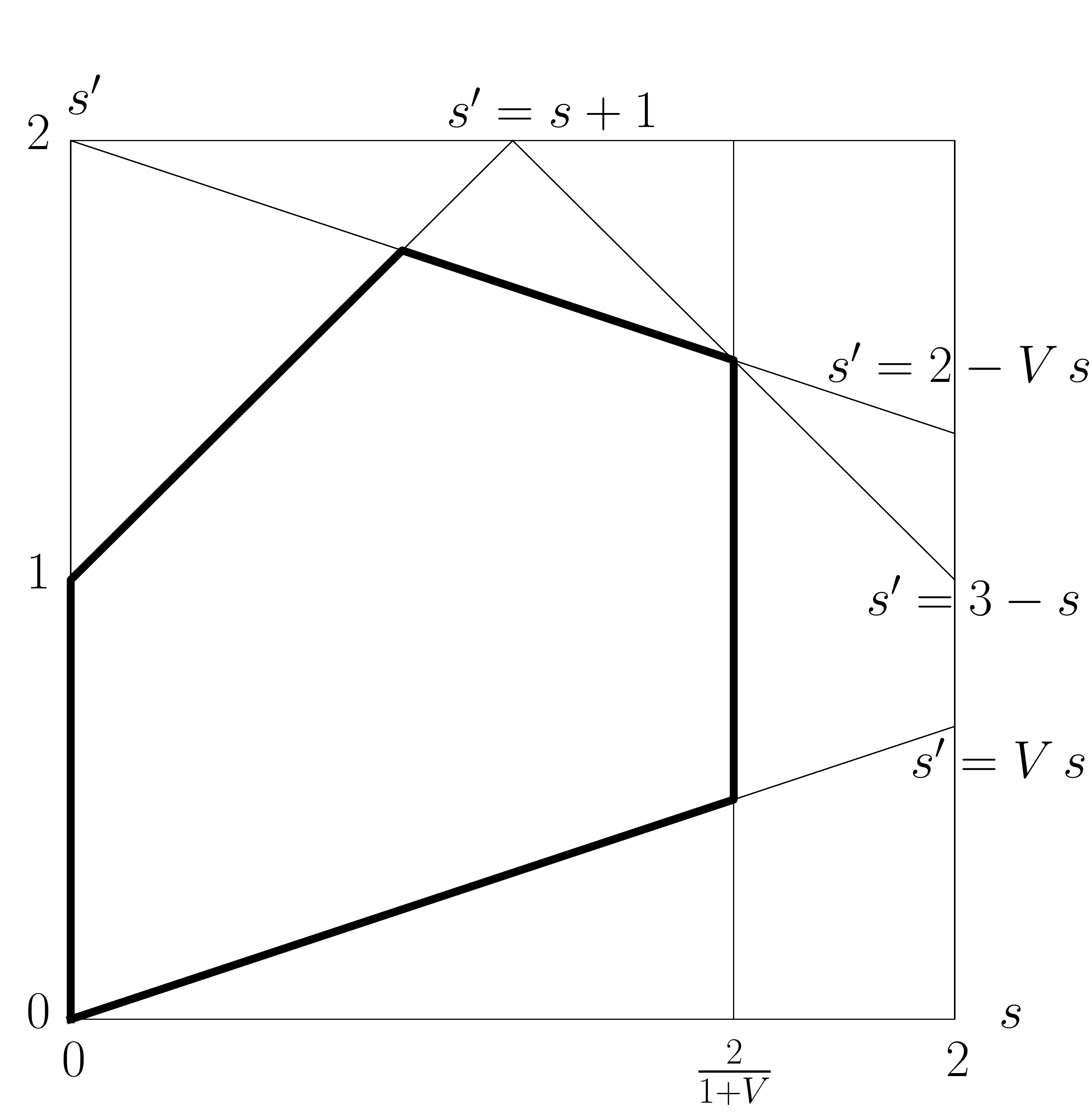}
  \hfill 
  \includegraphics[width=.49 \textwidth, angle=0] {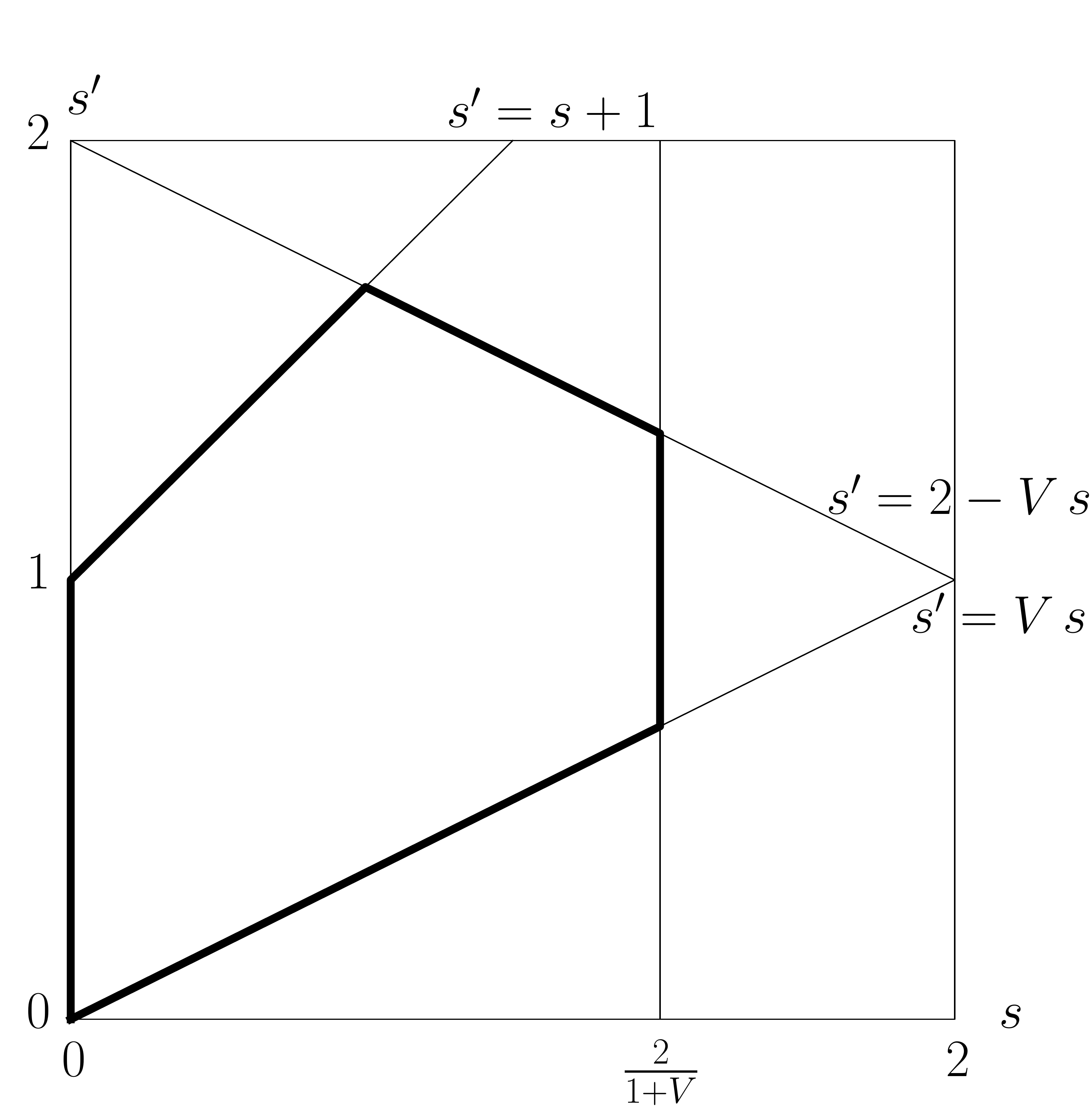}
  \includegraphics[width=.49 \textwidth, angle=0] {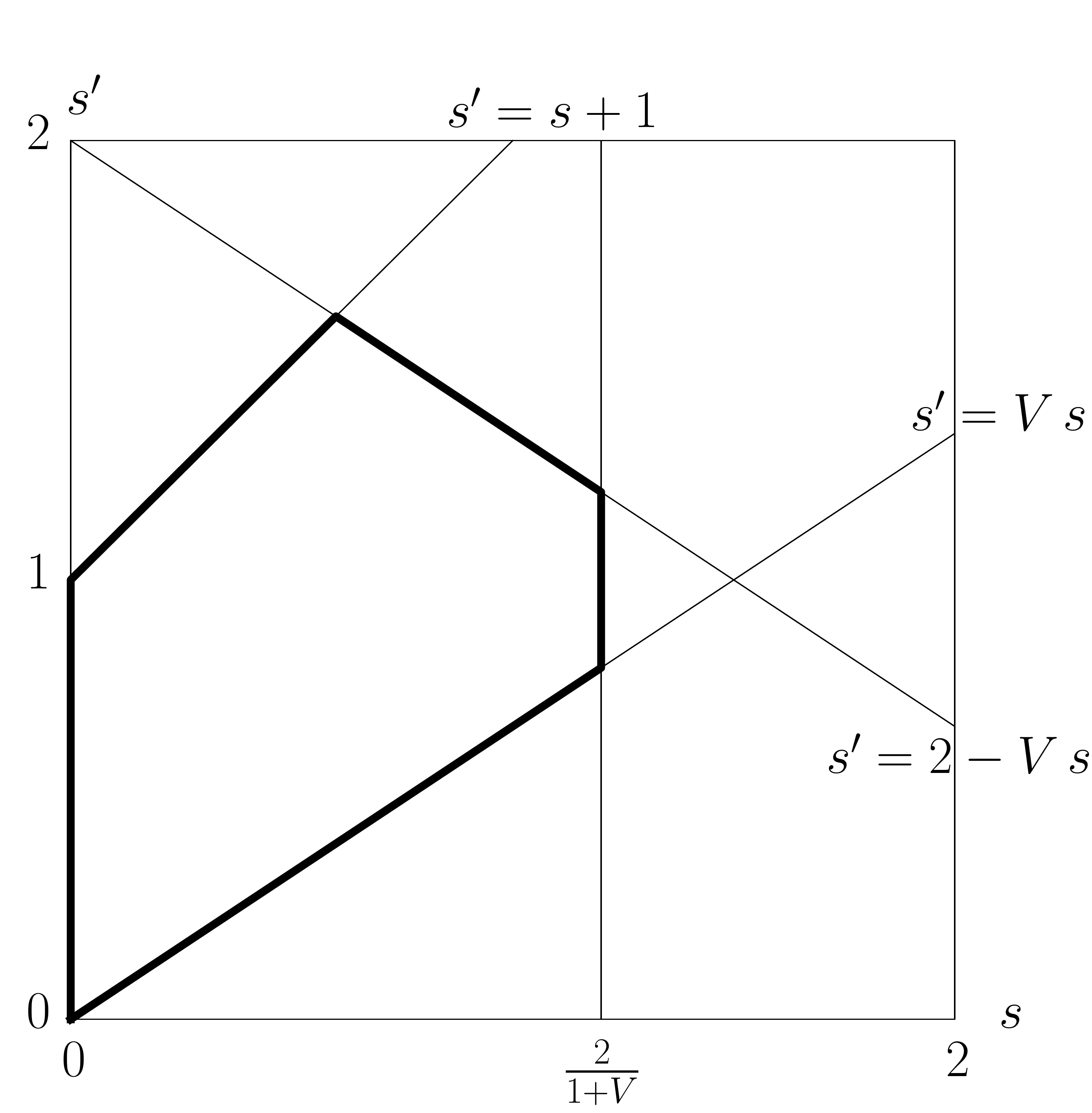}
  \hfill 
  \includegraphics[width=.49 \textwidth, angle=0] {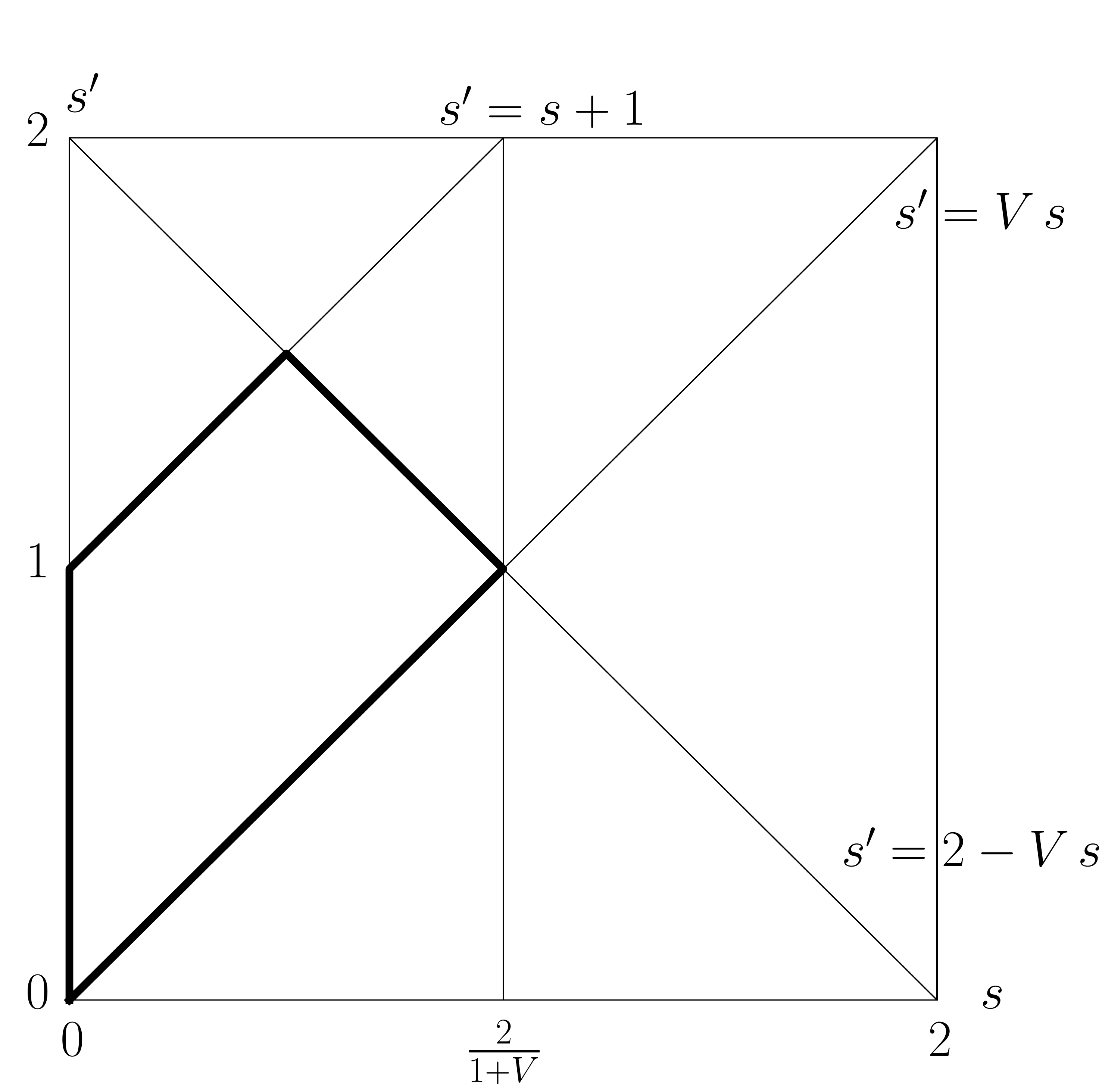}
  \caption{Necessary stability regions given by inequalities \eqref{necessary} for \(V=0\) (top-left), \(V=1/4\) (top-right), \(V=1/3\) (middle-left), \(V=1/2\) (middle-right), \(V=2/3\) (bottom-left), \(V=1\) (bottom-right).}
  \label{cn-fig} 
\end{figure}

\subsection{Numerical study of necessary and sufficient conditions for stability}

We now illustrate the necessary and sufficient stability regions for \(V=0\), \(V=1/4\), \(V=1/3\), \(V=1/2\), \(V=2/3\), and $V=1$ for various ranges of $u$ in the Figure~\ref{cns-numeric-fig}.

\begin{figure}[H]
  \includegraphics[width=.45 \textwidth, angle=0] {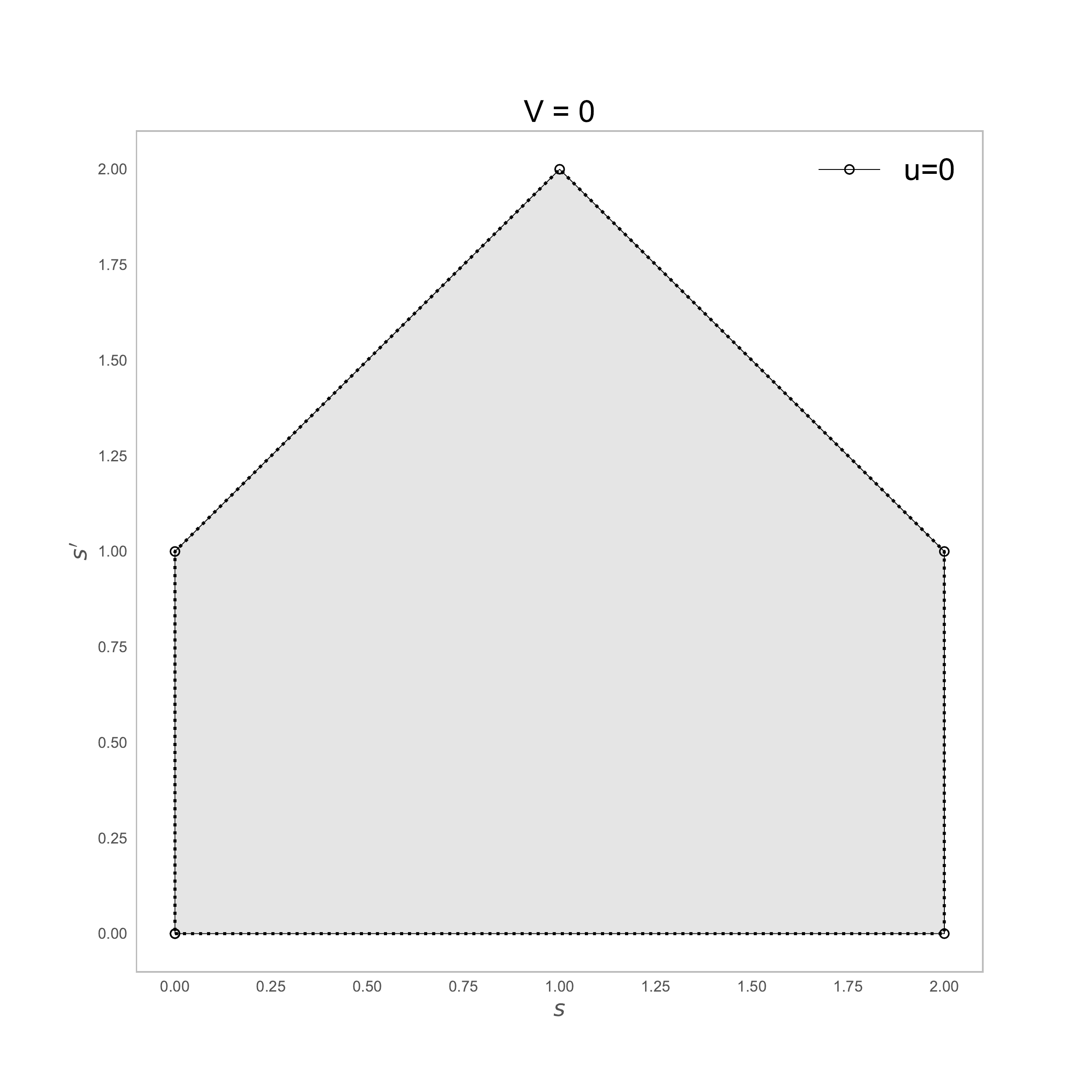}
  \hfill
  \includegraphics[width=.45 \textwidth, angle=0] {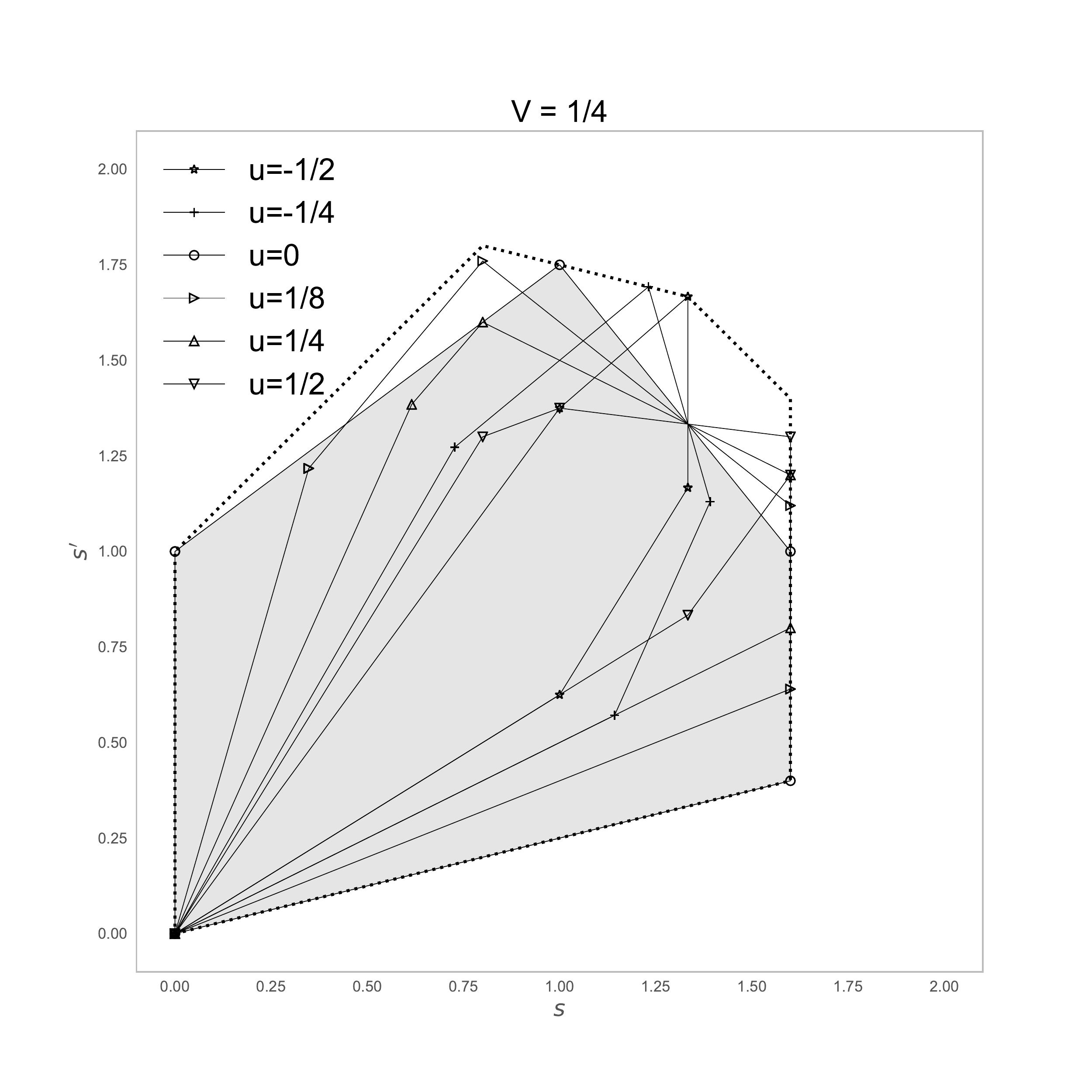}
  \includegraphics[width=.45 \textwidth, angle=0] {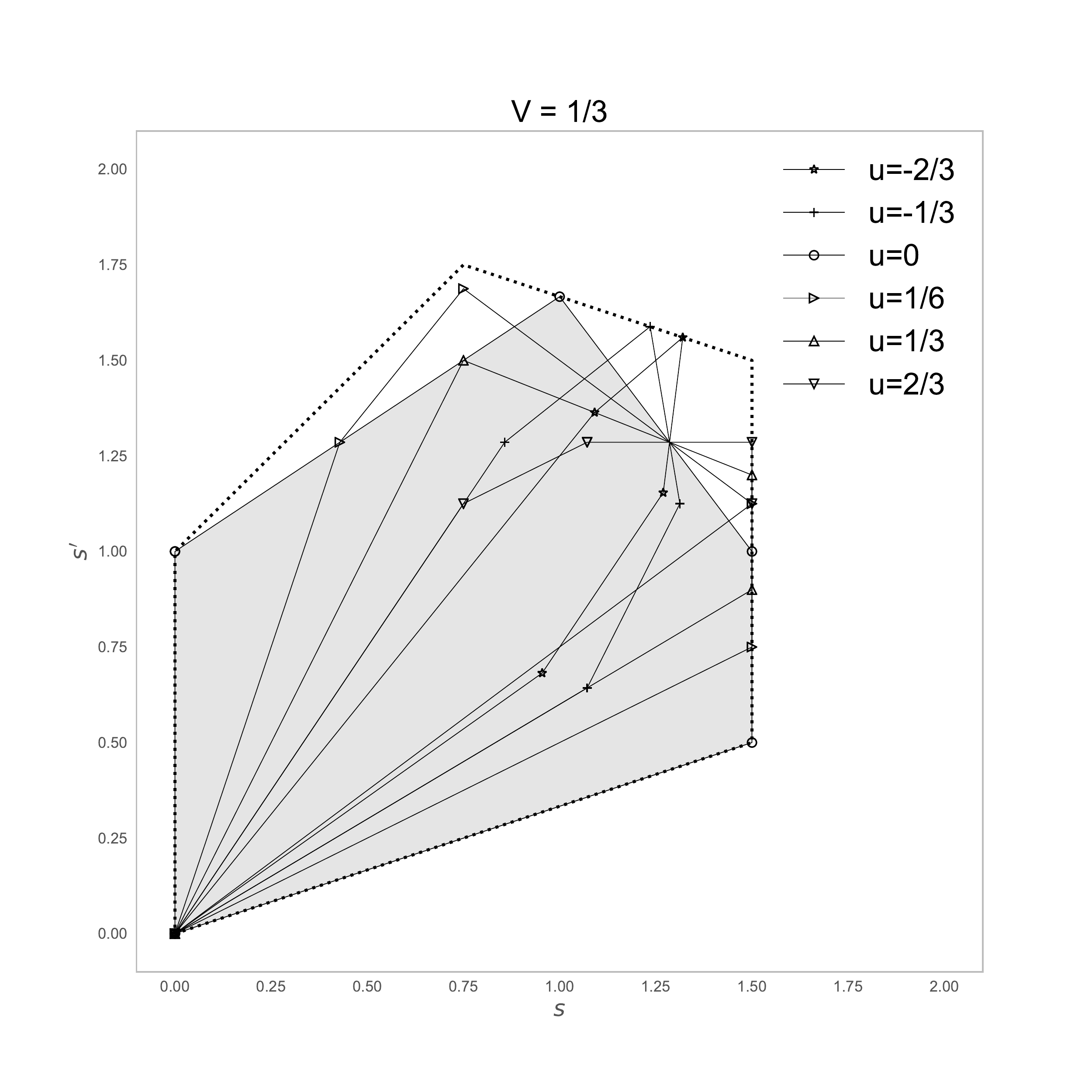}
  \hfill
  \includegraphics[width=.45 \textwidth, angle=0] {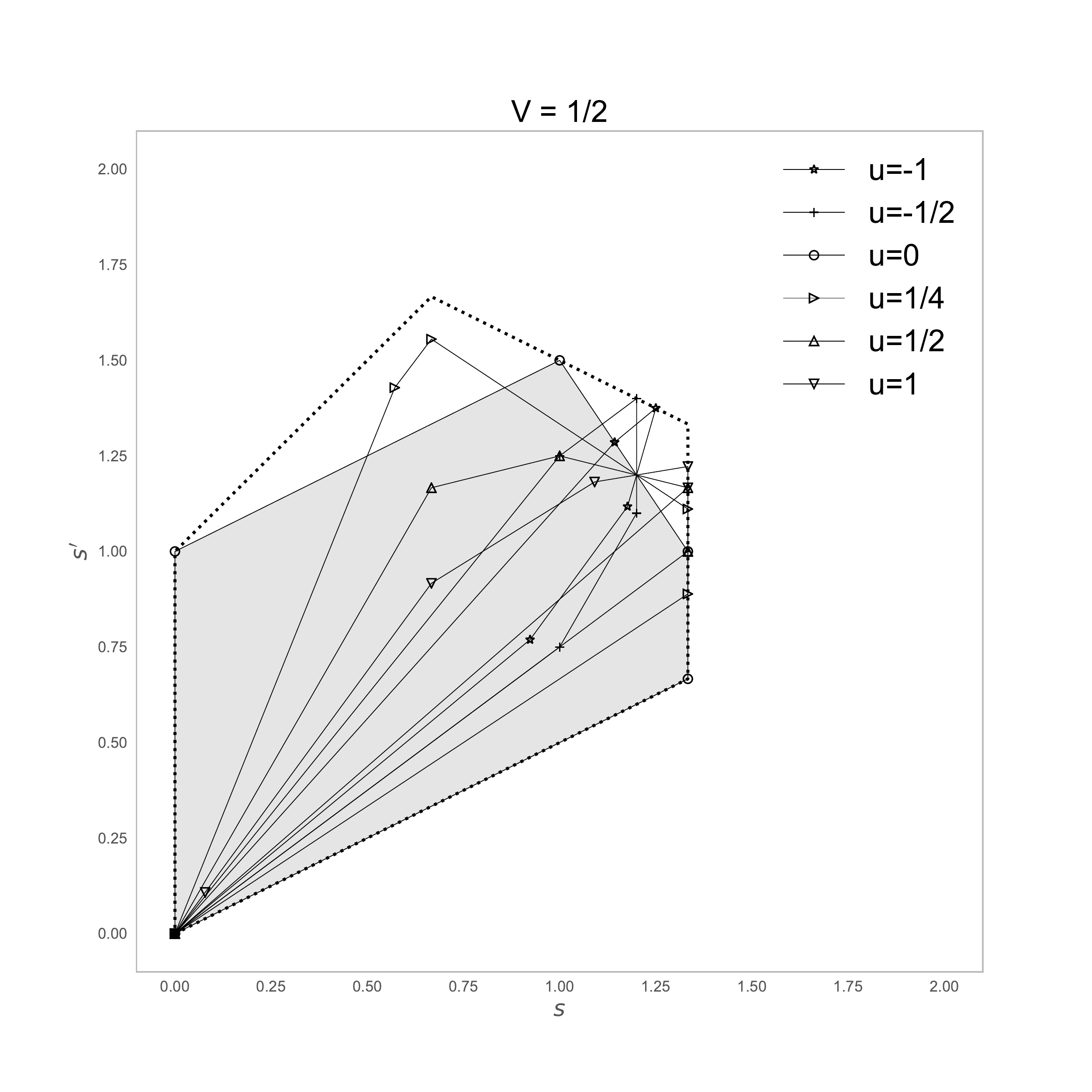}
  \includegraphics[width=.45 \textwidth, angle=0] {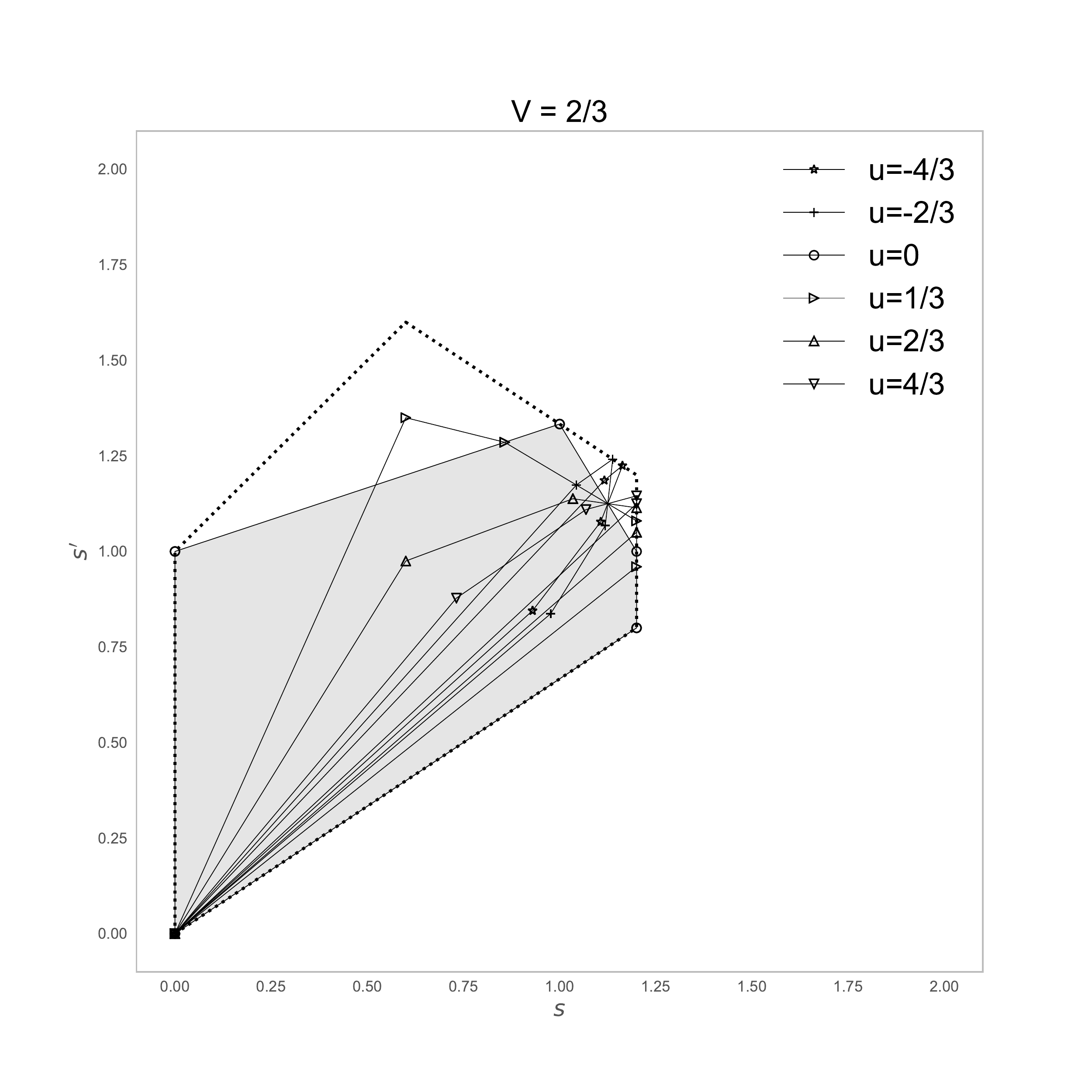}
  \hfill
  \includegraphics[width=.45 \textwidth, angle=0] {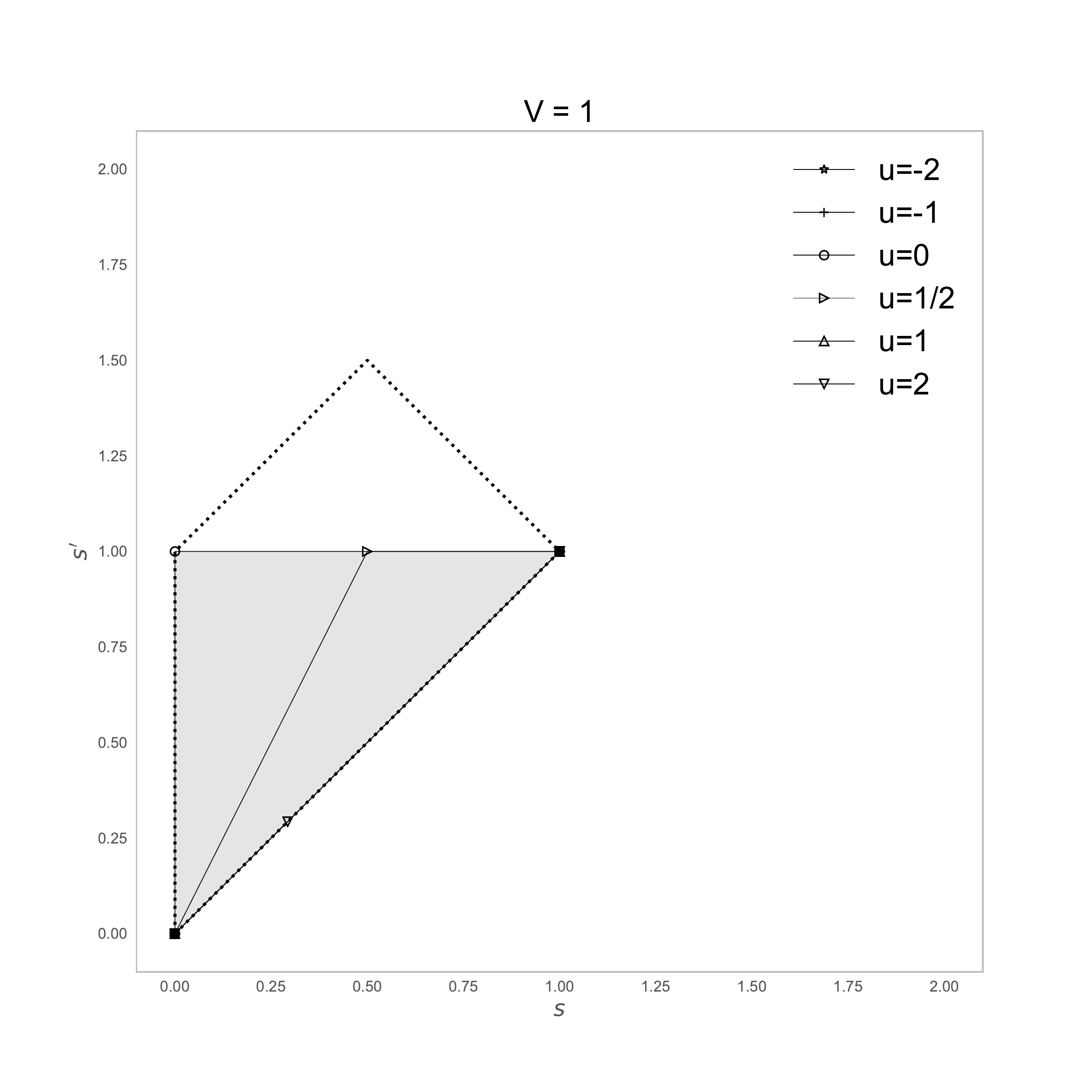}
  \caption{Nunerical study of necessary and sufficient stability regions}
  \label{cns-numeric-fig}
\end{figure}

For each value of the velocity \(V\), the necessary and sufficient stability regions given by the relations \eqref{cns-stabilite} are displayed for several values of the relative velocity: \(u\in\lbrace -2V, -V, 0, V/2, V, 2V\rbrace\). In dotted line, the necessary stability region of the Proposition~\ref{th:necessary} is added for remind. The particular value \(u=0\) is enhanced by filling the region in gray.

Some analysis can be drawn from these figures:
\begin{itemize}
  \item the stability region changes with the relative velocity;
  \item the maximal value of the first relaxation parameter \(s\) is obtained (not only) for \(u=0\);
  \item the stability region is not clearly more interessant (larger or including greater value of \(s\)) for \(u=V\);
  \item the point \((s,s')=(1,1)\) is always in the stability region.
\end{itemize}

To conclude this section, this notion of stability allows a large set of values for the relaxation parameters. If the scheme is used to simulate the hyperbolic advection equation without diffusive term, we can try to minimize the numerical diffusion while maintaining this stability property.
This task is complicated and out of the scope of the paper as the second-order term reads as a non-linear formula that links all the parameters \(s\), \(\alpha\), and \(V\).



\section{Numerical illustrations}
\label{sec:num}

In this section, we illustrate the stability property with numerical simulations involving the D1Q3 model with and without relative velocity used to simulate the linear advection equation. The stability is demonstrated for the parameters chosen according to the analysis presented in the previous sections. Oscillations are seen whenever the parameters go beyond the stability limits presented, as highlighted in the results.

The parameters chosen for the simulations are the following:

\begin{center}
  \begin{tabular}{p{3.5cm} p{.65cm} p{.65cm} p{.5cm} p{.5cm} p{3.5cm}}
    \toprule
    &$V$&$u$&$s$&$s'$&$\alpha$\\
    \midrule
    left  (stable)
    & 0.25 & 0.0  & 1.6 & 1.3 &  0.3076923076923076\\
    & 0.25 & 0.25 & 1.6 & 1.3 & -0.17548076923076938\\
    \midrule
    right (unstable)
    & 0.25 & 0.0  & 1.9 & 1.4 &  0.14285714285714302\\
    & 0.25 & 0.25 & 1.9 & 1.4 & -0.10491071428571441\\ 
    \bottomrule
  \end{tabular}
\end{center}

For the left figures, the parameters are chosen in order to satisfy the stability property. We observe numerically also  a maximum principle, even if it is not formally the stability notion that we investigate.
For the right figures, the parameters are chosen in order not to break the stability property. 

\begin{figure} [H]
    \includegraphics[width=.49 \textwidth ,angle=0]{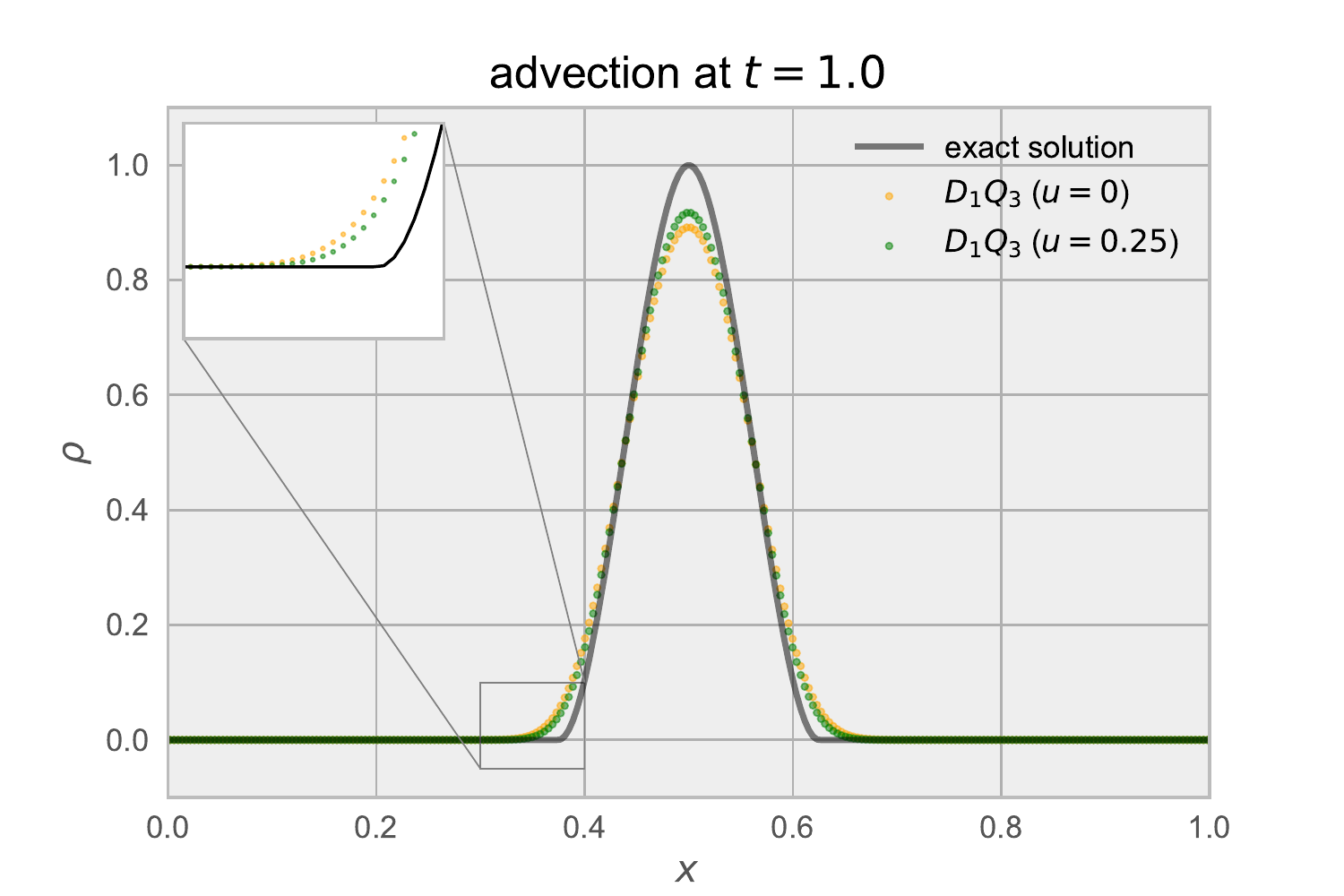} \hfill  
    \includegraphics[width=.49 \textwidth ,angle=0]{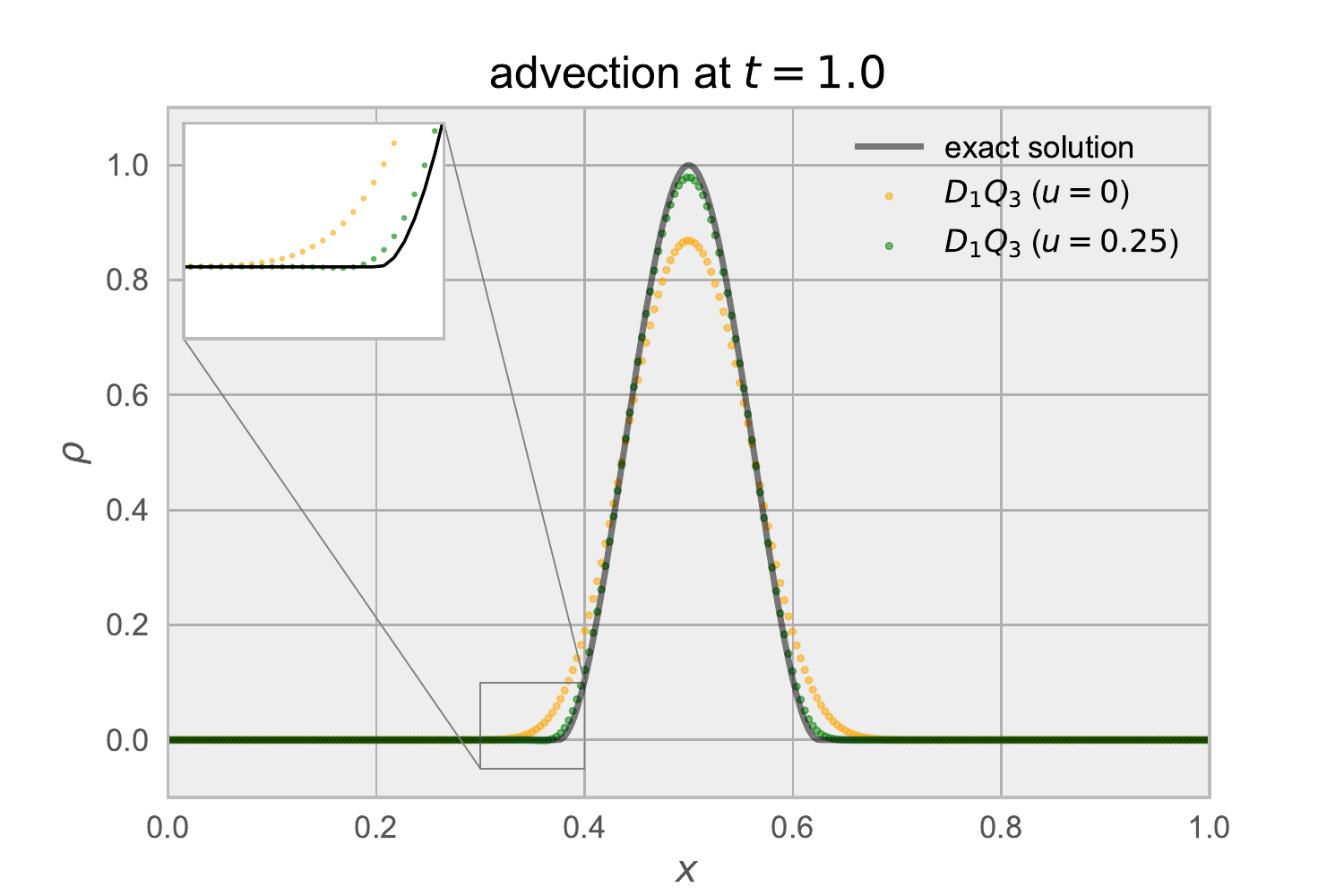}
    \caption{Smooth profile with a continuous derivative. The parameters for the D1Q3 are tuned in order to have (left)
    or not (right)   the non-negativity property} 
  \label{fig-numa}
\end{figure} 

\begin{figure} [H]
    \includegraphics[width=.49 \textwidth ,angle=0]{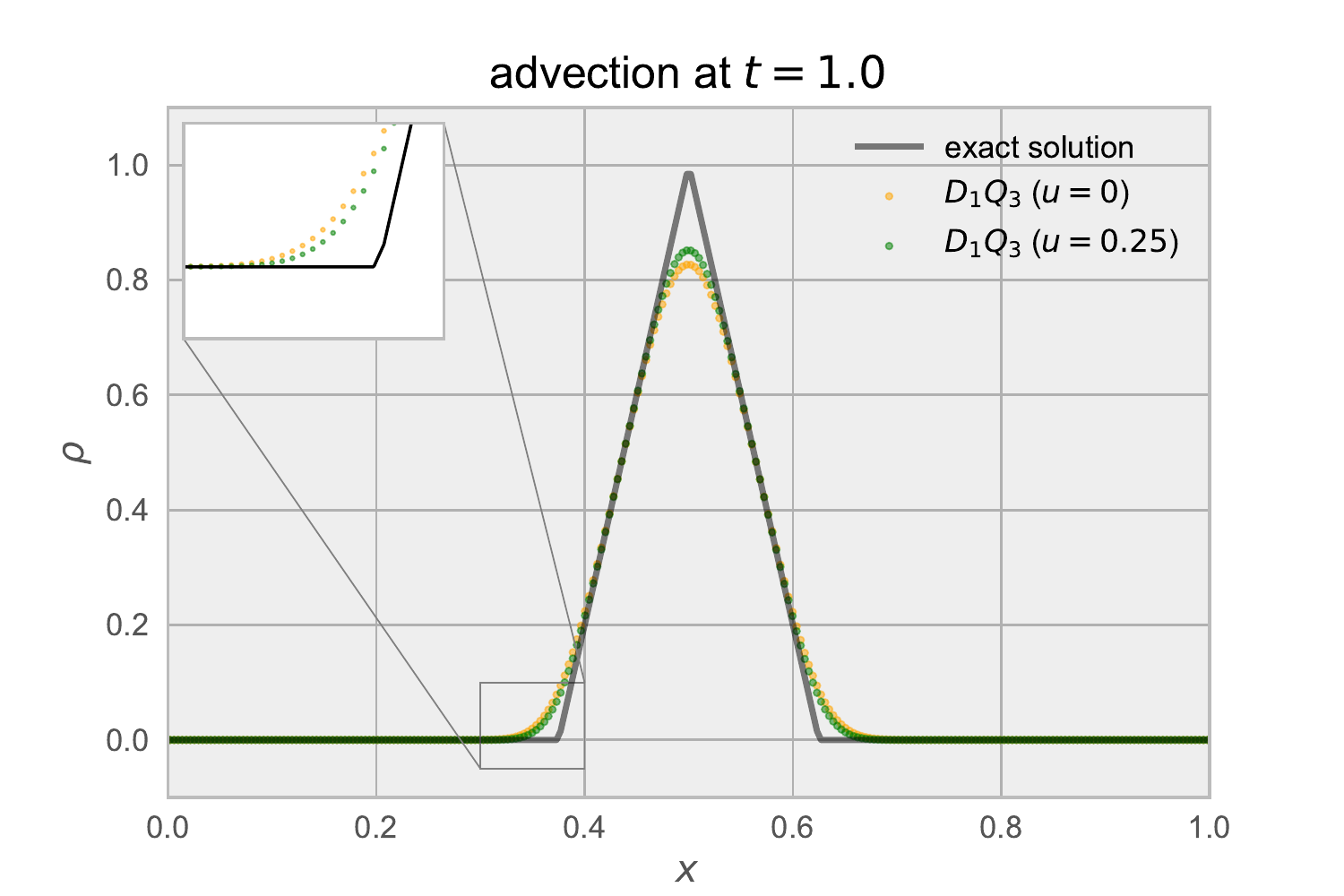} \hfill
    \includegraphics[width=.49 \textwidth ,angle=0]{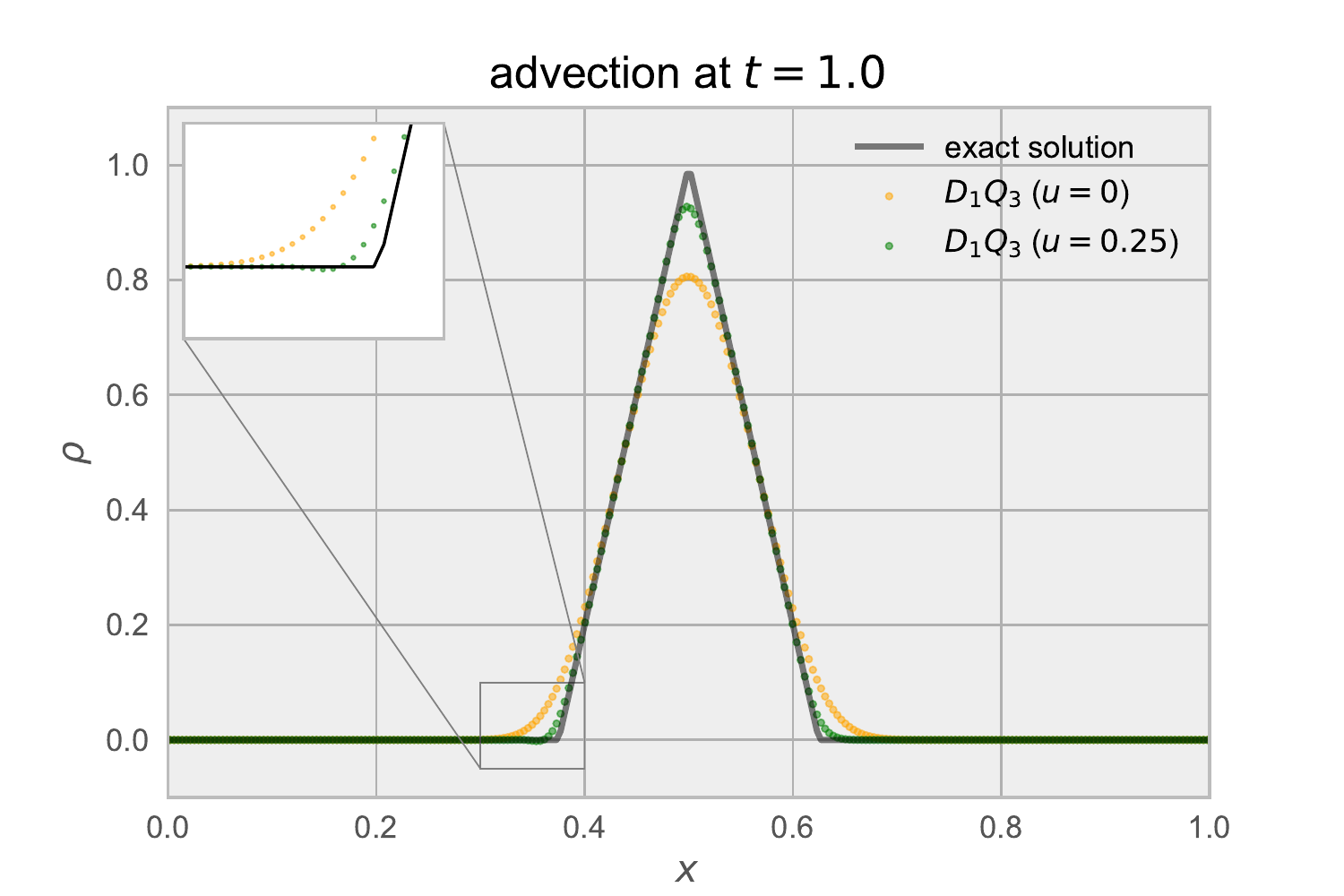}  
    \caption{Continuous profile. The parameters for the D1Q3 are tuned in order to have (left)
    or not (right)   the non-negativity property} 
  \label{fig-numb}
\end{figure} 

\begin{figure}[H]
    \includegraphics[width=.49 \textwidth ,angle=0]{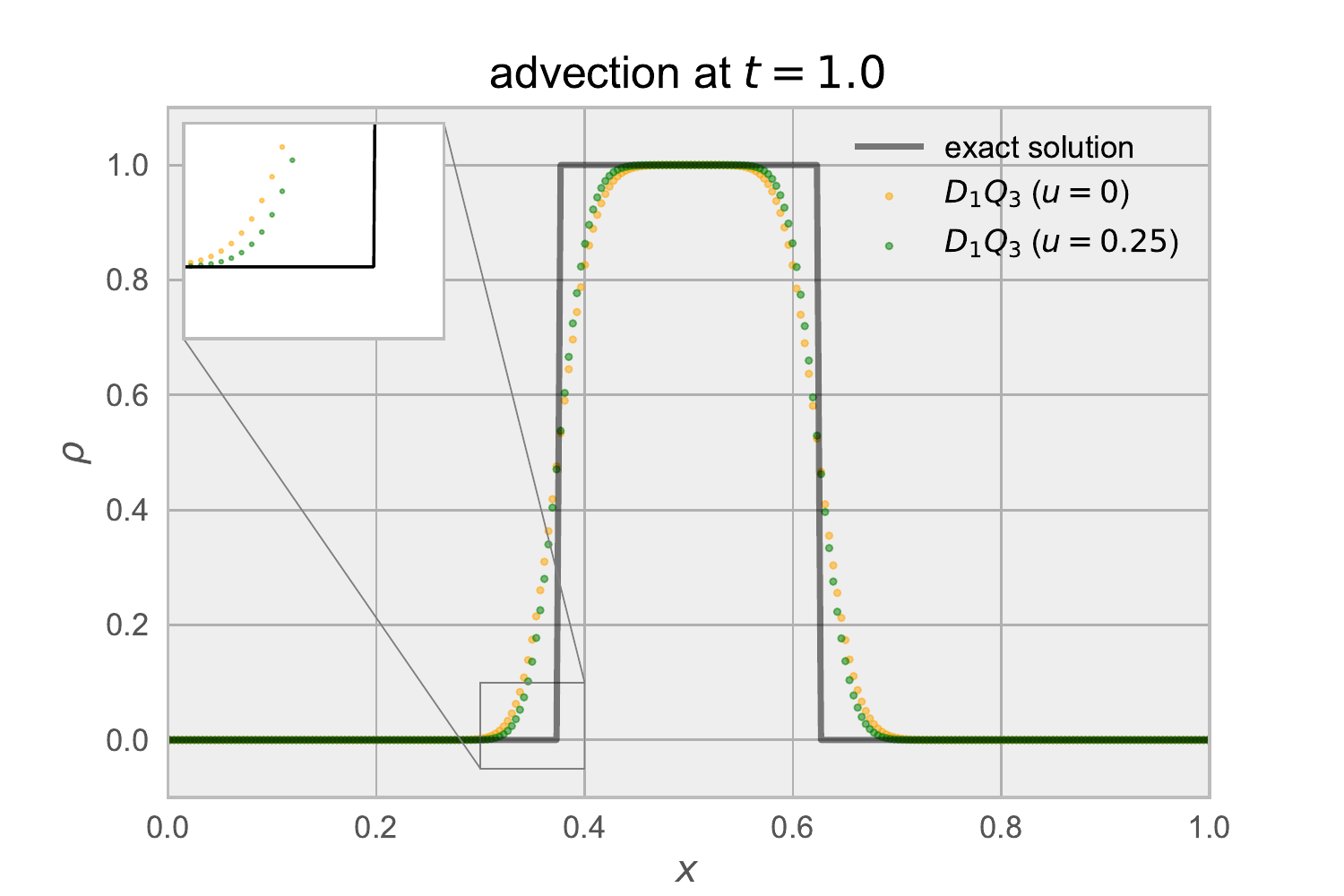} \hfill
    \includegraphics[width=.49 \textwidth ,angle=0]{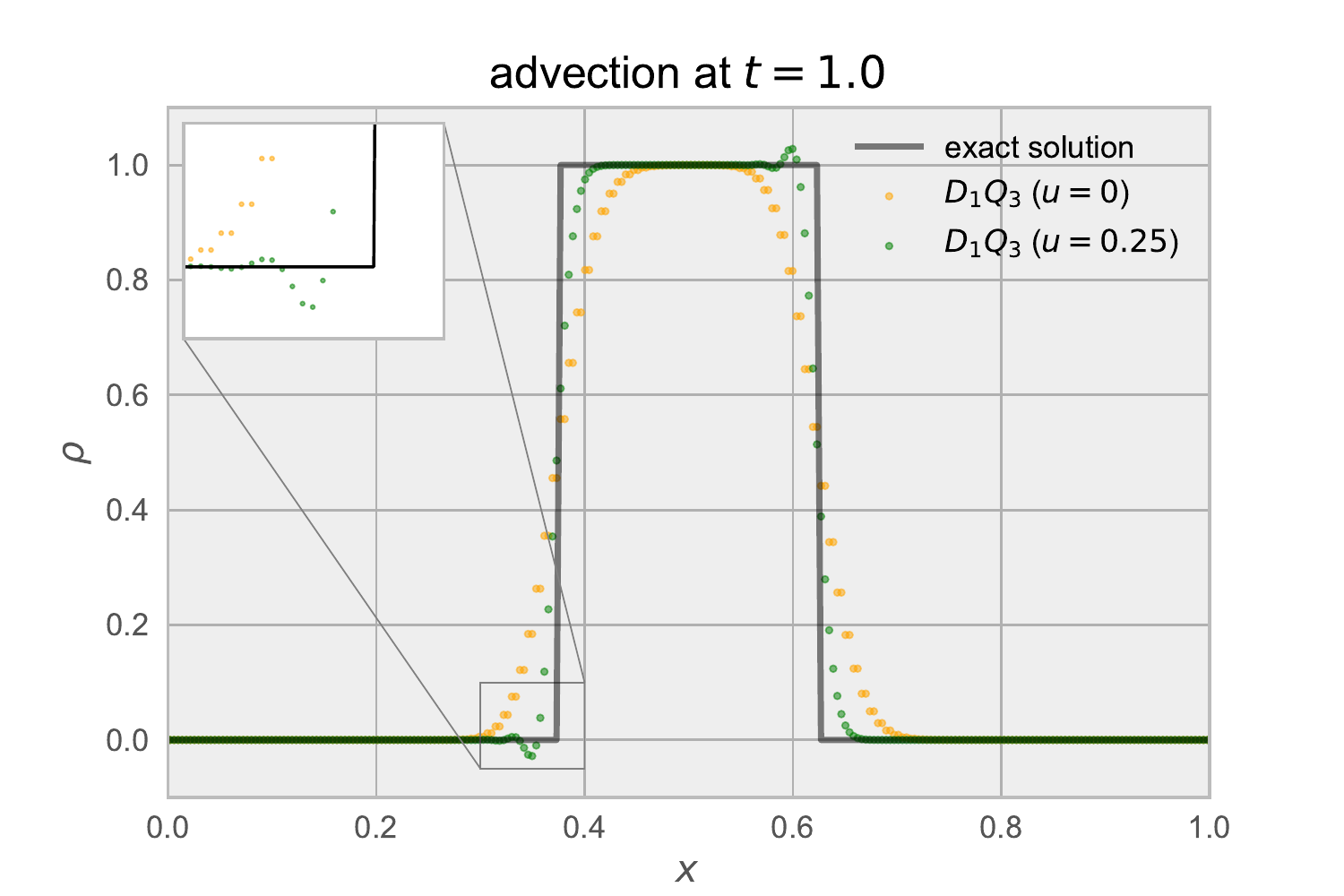} 
    \caption{Discontinuous profile. The parameters for the D1Q3 are tuned in order to have (left)
    or not (right)   the non-negativity property} 
  \label{fig-numc}
\end{figure} 

If the profile is smooth, we have observed that no numerical oscillations occur  even if the non-negativity property of the matrix is not satisfied. If the profile is just continuous, small negative values of the macroscopic quantity  are observed when our non-negativity property is not satisfied. Last but not least, classical oscillations are visible for discontinuous profiles if our non-negativity property is not satisfied. These  oscillations are eliminated when the non-negativity property of the matrix is realized.



\section{Conclusion}

In this contribution, we have investigated a stability property for a classical mono-dimensional linear three velocities lattice Boltzmann scheme with relative velocity. This property ensures that the particle distribution functions remain nonnegative if they are at initial time. 
We then give a necessary and sufficient condition to describe the stability region. The case without relative velocity is completely describe and simplest necessary conditions are given for the general case. 
We finally propose some numerical simulations that illustrate the stability property: even if the stability notion that we investigate is not exactly a cosntraint of convexity, a numerical maximum principle is observed if the parameters are in the stability region whereas numerical oscillations appear (in particular for non smooth profils) if the parameters are not there.

Moreover, relative velocities modify the stability array in a nontrivial manner. For instance, intuition might have suggested that the stability region for the relative velocity equal to the advection velocity contains all the others but it is realy not the case. For a given advection velocity, relative velocities cannot be used to increase the value of the first relaxation parameter, the one involved in the numerical diffusion.

The non-negativity of the relaxation matrix could be extended to nonliear schemes. The theoretical study will be much more technical and has not been performed. Nevertheless, numerical experiments for the Burgers equation show that the behavior of the D1Q3 scheme, and in particular the possibility or not of oscillations, is analogous to the linear case.


\bibliographystyle{plain}
\bibliography{Dubois_Graille_Rao}

\begin{thebibliography}{10}

\bibitem{caetano_result_2019}
F.~Caetano, F.~Dubois, and B.~Graille.
\newblock A result of convergence for a mono-dimensional two-velocities lattice
  {Boltzmann} scheme.
\newblock {\em arXiv:1905.12393 [math]}, May 2019.
\newblock arXiv: 1905.12393.

\bibitem{dellar_nonhydrodynamic_2002}
P.~J. Dellar.
\newblock Nonhydrodynamic modes and a priori construction of shallow water
  lattice {Boltzmann} equations.
\newblock {\em Physical Review E}, 65(3):036309, February 2002.

\bibitem{dhumiere_generalized_1992}
D.~d'Humi\`eres.
\newblock Generalized {Lattice}-{Boltzmann} equations.
\newblock In {\em Rarefied {Gas} {Dynamics}: {Theory} and simulation}, volume
  159, pages 450--458. AIAA Progress in astronomics and aeronautics, 1992.

\bibitem{dubois_lattice_2015}
F.~Dubois, T.~F\'evrier, and B.~Graille.
\newblock Lattice {Boltzmann} schemes with relative velocities.
\newblock {\em Communications in Computational Physics}, 17(4):1088--1112,
  2015.

\bibitem{dubois_stability_2015}
F.~Dubois, T.~F\'evrier, and B.~Graille.
\newblock On the stability of a relative velocity lattice {Boltzmann} scheme
  for compressible {Navier}–{Stokes} equations.
\newblock {\em Comptes Rendus M\'ecanique}, 343(10-11):599--610, 2015.

\bibitem{dubois-lallemand_2009}
F.~Dubois and P.~Lallemand.
\newblock Towards higher order lattice {Boltzmann} schemes.
\newblock {\em Journal of Statistical Mechanics: Theory and Experiment},
  2009:P06006, 2009.

\bibitem{geier_cascaded_2006}
M.~Geier, A.~Greiner, and J.~C. Korvink.
\newblock Cascaded digital lattice {Boltzmann} automata for high reynolds
  number flow.
\newblock {\em Physical Review E}, 73:066705, 2006.

\bibitem{ginzburg_truncation_2012}
I.~Ginzburg.
\newblock Truncation {Errors}, {Exact} {And} {Heuristic} {Stability} {Analysis}
  {Of} {Two}-{Relaxation}-{Times} {Lattice} {Boltzmann} {Schemes} {For}
  {Anisotropic} {Advection}-{Diffusion} {Equation}.
\newblock {\em Communications in Computational Physics}, 11(5):1439--1502, May
  2012.

\bibitem{lallemand_theory_2000}
P.~Lallemand and L.-S. Luo.
\newblock Theory of the lattice {Boltzmann} method: dispersion, dissipation,
  isotropy, {Galilean} invariance, and stability.
\newblock {\em Physical Review E}, 61:6546--6562, 2000.

\bibitem{osher_muscl_1985}
S.~Osher.
\newblock Convergence of generalized muscl schemes.
\newblock {\em SIAM Journal on Numerical Analysis}, 22(5):947--961, 1985.

\end{thebibliography}

\end{document}